\documentclass[pagebackref,colorlinks,citecolor=blue,linkcolor=blue,urlcolor=blue,filecolor=blue]{article}
\pdfoutput=1
\usepackage{geometry}
\usepackage{times}
\usepackage{tikz-cd}
\usetikzlibrary{calc}
\usepackage{float}
\usepackage{amsmath}
\usepackage{amsfonts}
\usepackage{amsthm}
\usepackage{enumitem}
\usepackage{dsfont}
\usepackage{amssymb}
\usepackage{pifont}
\usepackage{mathtools}
\usepackage{comment}
\usepackage{graphicx, caption} 
\usepackage{float}
\usepackage{xcolor}
\usetikzlibrary{matrix}
\usepackage[T1]{fontenc}
\usepackage{tgpagella}
\usepackage[utf8]{inputenc}
\usepackage{pinlabel} 

\usepackage{hyperref}

\newtheorem*{T1}{Theorem~\ref{point is square}}
\newtheorem*{T2}{Theorem~\ref{square is discrete}}
\newtheorem*{T3}{Theorem~\ref{point is discrete}}
\newtheorem*{T4}{Theorem~\ref{Abrams theorem}}

\newtheorem{thm}{Theorem}[section]
\newtheorem{lem}[thm]{Lemma}
\newtheorem{prop}[thm]{Proposition}
\newtheorem{conj}[thm]{Conjecture}
\newtheorem{cor}[thm]{Corollary}
\newtheorem{exam}[thm]{Example}
\newtheorem{question}[thm]{Question}
\theoremstyle{definition}

\theoremstyle{definition}

\newtheorem{defn}[thm]{Definition}

\newtheorem{remark}[thm]{Remark}

\newtheorem*{claim*}{Claim}
\newtheorem*{quest*}{Question}
\newtheorem*{remark*}{Remark}
\newtheorem*{fact*}{Fact}

\newcommand{\thmtext}{
For $m\le n$, $g\ge 1$, and $b=0$ or $1$, let $K_{g,b}(n)$ be the cube complex defined in Definition \ref{K(n) definition}.
Then, if $d\le1$,
\[
SF_{m}\big(K_{g,b}(n),d\big)\simeq F_{m}\big(K_{g,b}(n)\big)\simeq F_{m}(\Sigma_{g,b}).
\]
}

\newcommand{\thmtextone}{
For $m\le n$, the subcomplex $DF_{m}\big(K_{g,0}(n)\big)$ is a deformation retract of $SF_{m}\big(K_{g,0}(n)\big)$; similarly, the subcomplex $DF_{m}\big(K^{*}_{g,1}(n)\big)$ is a deformation retract of $SF_{m}\big(K_{g,1}(n)\big)$.
}

\newcommand{\thmtexttwo}{
Let $\Sigma_{g,b}$ be a surface of genus $g\ge 1$ with $b=0$ or $1$ boundary components.
If $m\le n$, then 
\[
F_{m}(\Sigma_{g,0})\simeq DF_{m}\big(K_{g,0}(n)\big)\indent\text{and}\indent F_{m}(\Sigma_{g,1})\simeq DF_{m}\big(K^{*}_{g,1}(n)\big),
\]
where $K_{g,0}(n)$ and $K^{*}_{g,1}(n)$ are the square-tiled surfaces defined in Definition \ref{K(n) definition}.
}

\newcommand{\thmabrams}{
(Abrams \cite[Theorem 2.1]{abrams2000configuration})
Let $\Gamma$ be a connected graph, i.e., a $1$-dimensional cube complex, with at least $n$ vertices.
Then $F_{n}(\Gamma)$ deformation retracts onto $DF_{n}(\Gamma)$ if
\begin{enumerate}
    \item each path connecting distinct essential vertices of $\Gamma$ has length at least $n+1$, and
    \item each homotopically essential path connecting a vertex to itself has length at least $n+1$.
\end{enumerate}
}

\newcommand{\Z}{\ensuremath{\mathbb{Z}}}

\newcommand{\R}{\ensuremath{\mathbb{R}}}

\title{A discrete model for surface configuration spaces}
\author{Nicholas Wawrykow}
\date{}

\begin{document}
\maketitle
\begin{abstract}
One of the primary methods of studying the topology of configurations of points in a graph and configurations of disks in a planar region has been to examine discrete combinatorial models arising from the underlying spaces.
Despite the success of these models in the graph and disk settings, they have not been constructed for the vast majority of surface configuration spaces.
In this paper, we construct such a model for the ordered configuration space of $m$ points in an oriented surface $\Sigma$.
More specifically, we prove that if we give $\Sigma$ a certain cube complex structure $K$, then the ordered configuration space of $m$ points in $\Sigma$ is homotopy equivalent to a subcomplex of $K^{m}$.
\end{abstract}

\section{Introduction}

We introduce a discrete combinatorial model for the ordered configuration space of points in an oriented surface.
Namely, we show that if we give an oriented surface $\Sigma_{g,b}$ of genus $g$ with $b=0$ or $1$ boundary components the cube complex structure $K_{g,b}(n)$ defined in Definition \ref{K(n) definition}, then, for $m\le n$, the ordered configuration space of $m$ points in $\Sigma_{g,b}$, denoted $F_{m}(\Sigma_{g,b})$, is homotopy equivalent to a closed subcomplex of $\big(K_{g,b}(n)\big)^{m}$ called the \emph{$m^{\text{th}}$-discrete ordered configuration space} of $K_{g,b}(n)$.
Given a cellular complex $X$, this discrete configuration space, denoted $DF_{m}(X)$, is the subcomplex of $X^{m}$ consisting of $m$-fold products of cells in $X$ with disjoint closures, i.e.,
\[
DF_{m}(X):=\left\{\sigma_{1}\times \cdots\times \sigma_{m}\in X^{m}|\overline{\sigma_{i}}\cap \overline{\sigma_{j}}=\emptyset\text{ if }i\neq j\right\}.
\]
That is, we prove the following theorem describing an explicit $S_{m}$-equivariant CW-complex structure for $F_{m}(\Sigma_{g,b})$.

\begin{thm}\label{point is discrete}
\thmtexttwo
\end{thm}

This model has two conspicuous upsides: First, it is amenable to the methods of discrete Morse theory, which have been used to great effect in the study of configurations of points in graphs and configurations of disks in planar regions; second, it allows one to easily visualize representatives of (co)homology classes in terms of the motion of particles about the surface.

Our method of constructing a model for configuration space is qualitatively different from previous approaches in that it is fundamentally geometric.
Our model, which takes its inspiration from Abrams's discrete model for graph configuration spaces, is constructed by considering disk configuration spaces.
Historically, the studies of graph configuration spaces, disk configuration spaces, and manifold configuration spaces have largely been disjoint; in this paper, we intend to connect these areas of research. 

\emph{A (very) brief history:} The study of configurations of points in a surface began over a century ago with Hurwitz's \cite{hurwitz1891riemann} and Artin's \cite{artin1925isotopie, artin1947theory} work on the (pure) braid groups.
Since then, configuration spaces have been studied via a wide variety of techniques.
Fadell and Neuwirth used fibrations to prove that surface configuration spaces are classifying spaces for the surface braid groups \cite[Theorem 1 and Corollary 2.2]{fadell1962configuration}.
Arnol'd used differential forms \cite{arnold1969cohomology} and Cohen used iterated loop spaces \cite{cohen2007homology} to find presentations for the cohomology ring of the ordered configuration space of points in the plane; McDuff \cite{mcduff1975configuration} and Segal \cite{segal1979topology} built on these ideas to prove that the homology groups of the unordered configuration space of a connected non-compact manifold stabilize as the number of points tends to infinity.
Church, Ellenberg, and Farb used spectral sequences and combinatorial categories to prove a representation-theoretic version of these stability results for the homology of the ordered configuration space of an oriented non-compact finite type manifold of dimension at least $2$ \cite{church2012homological, church2014fi, church2015fi}.
Despite these results---and the numerous others we omit for the sake of brevity---much is still unknown about surface configuration spaces, e.g., most of the Betti numbers of the ordered space of the torus remain unknown \cite[Conjecture 2.10]{pagaria2022asymptotic}.
As a result, we seek new ways to approach these spaces.

In the last thirty years, one approach to the problem of finding new methods has been to construct models for configuration spaces.
This paper is another contribution in this vein.
Totaro \cite{totaro1996configuration} and Kriz \cite{kriz1994rational} gave a model that corresponds to the $E^{2}$-page of a certain spectral sequence.
Alternatively, one can construct the Salvetti complex \cite{salvetti1994homotopy} for these spaces using the presentations for the surface pure braid groups given in \cite{bellingeri2001presentation}.
More recently, Campos, Idrissi, Lambrechts, and Willwacher used commutative differential graded algebras arising from Kontsevich's graph complexes to give real models for ordered configuration spaces, e.g., \cite{idrissi2019lambrechts, campos2019configuration, campos2018configuration,campos2023model}.
Though these models have provided great insight, there is still much to be learned about configuration spaces, prompting us to look elsewhere for inspiration.

\subsection{Connections to graph and square configuration spaces}

To prove Theorem \ref{point is discrete} we avoid most of the classical techniques used in the study of surface configuration spaces.
Instead, we draw much of our inspiration from two newer areas of research that are studied via broadly different methods, namely configurations of points in graphs and configurations of disks in planar regions.

Unlike manifold configuration spaces, the study of graph configuration spaces started merely twenty-five years ago with the work of Ghrist \cite{Ghr01} and Abrams \cite{abrams2000configuration}.
Using the underlying cube complex structure of a graph, Ghrist proved that graph configuration spaces are classifying spaces for the graph braid groups.
Abrams used this inherent cube complex structure in a slightly different way to give graph configuration space the structure of a cube complex, showing that if a graph $\Gamma$ is sufficiently subdivided, then the $m^{\text{th}}$-ordered configuration space of $\Gamma$ deformation retracts onto to the $m^{\text{th}}$-discrete ordered configuration space of $\Gamma$.
That is, Abrams proved the following theorem.

\begin{thm}\label{Abrams theorem}
\thmabrams  
\end{thm}

Outside of the case $n=2$, Abrams was unable to write down the deformation retract implied by Theorem \ref{Abrams theorem}.
Instead, Abrams proved that $DF_{n}(\Gamma)$ is naturally a subspace of $F_{n}(\Gamma)$, and showed that both spaces are connected and have the same fundamental group.
Since Ghrist proved that $F_{n}(\Gamma)$ is a classifying space for the pure braid group on $n$ strands in $\Gamma$ \cite[Corollary 2.4]{Ghr01}, by Whitehead's Theorem it suffices to show that the higher homotopy groups of $DF_{n}(\Gamma)$ are trivial.
Abrams did this via a graphs of spaces argument that does not translate to the surface configuration space setting.
As such, we need further inspiration to prove Theorem \ref{point is discrete}.

The spark comes from the even newer field of disk configuration spaces.
These disk configuration spaces arise when we consider a metric space $(X, g)$ and insist that all the points in a configuration in $X$ are sufficiently far from each other and from the boundary.
They were first studied asymptotically in \cite{diaconis2009markov, carlsson2012computational, baryshnikov2014min} in the late 2000s and early 2010s, and in the last decade these spaces have received a flurry of interest, e.g., \cite{alpert2020generalized, alpert2021configuration, alpert2021configuration1,  alpert2024asymptotic, wawrykow2023representation, wawrykow2022On}.
In particular, if $X$ is a surface with well-defined notions of left and right, up and down, as is the case for translation surfaces $(X, g)$, then there is a reasonable notion of an $l^{\infty}$-ball, or square, in $X$.
In the context of this paper, the resulting \emph{ordered square configuration spaces $SF_{m}(X, d)$}, where
\[
SF_{m}\big(X, d\big):=\left\{(z_{1},\dots, z_{m})\in X^{m}|d_{\infty}(z_{i}, z_{j})\ge d\text{ for all }i\neq j\text{ and }d_{\infty}(z_{i},\partial X)\ge \frac{d}{2}\text{ for all }i\right\},
\]
are especially important, as they serve as the missing link between point and discrete surface configuration spaces.

The papers of Plachta \cite{plachta2021configuration} and Alpert, Bauer, Kahle, MacPherson, and Spendlove \cite{alpert2023homology}, which discuss configurations of squares in rectangles, are of special import as they inspire two theorems that combine to yield Theorem \ref{point is discrete}.
Plachta proved that the homotopy type of the configuration space of squares in a rectangle is solely dependent on the size $d$ and multiplicity $m$ of the squares.
By giving the surface $\Sigma_{g, b}$ the square-tiled structure $K_{g,b}(n)$ defined in Definition \ref{K(n) definition}, and carefully handling its genus and singularities, we are able to prove the following theorem, a surface version of one of Plachta's results \cite[Theorem 18]{plachta2021configuration}.

\begin{thm}\label{point is square}
\thmtext
\end{thm}

In Alpert, Bauer, Kahle, MacPherson, and Spendlove's paper they prove that the configuration space $m$ of squares in a rectangle can be discretized in a manner similar to graph configuration spaces \cite[Theorem 3.5]{alpert2023homology}.
We prove that this is also true for the configuration space of $m$ unit-squares in $K_{g,b}(n)$ provided $m\le n$.

\begin{thm}\label{square is discrete}
\thmtextone
\end{thm}

Plachta \cite{plachta2021configuration} and Alpert, Bauer, Kahle, MacPherson, and Spendlove \cite{alpert2023homology} heavily leaned on the fact that their underlying spaces are rectangles, hence can be embedded in $\R^{2}$, in the proofs of their theorems, something that is not true for all other surfaces.
To circumvent this issue, we use a version of polar coordinates to derive our results, using the fact that we can treat the sole singularity of $K_{g,0}(n)$ as if it were the origin.
Using these coordinates allows us to a find a flow on the ``tangent space'' of $K_{g,b}(n)$, a vital step in our proof of Theorem \ref{point is square}.
We also use these polar coordinates to find inequalities that allow us to determine if a square configuration lies in our discrete configuration space.

Beyond yielding Theorem \ref{point is discrete}, our geometric argument can be used to recover Theorem 2.1 of \cite{abrams2000configuration}, giving an explicit retraction of $F_{m}(\Gamma)$ onto $DF_{m}(\Gamma)$.
Moreover, we argue that it can be adapted to give discrete models for the configuration space of points in certain manifolds of dimension at least $3$.

\subsection{Outline}
We begin Section \ref{cube complexes} by recalling the definition of a cubical complex.
Next, we define the cube complex structure $K_{g,b}(n)$ on the surface $\Sigma_{g,b}$ that we use throughout the paper.
We make several remarks about the geometry of this complex and prove Lemma \ref{section of small set no boundary}, showing that we can parametrize small subspaces of $K_{g,b}(n)$ via a version of polar coordinates.
In Section \ref{square configuration spaces} we define our square configuration spaces by considering what we call the $l^{\infty}$-metric on $K_{g,b}(n)$.
Next, we recall Plachta's definition of the tautogical function $\theta_{m}$ on spaces with an $l^{\infty}$-metric, and determine its regular values on $\big(K_{g,b}(n)\big)^{m}$ via the parametrization given by Lemma \ref{section of small set no boundary}.
We use this to prove that the point configuration space of $\Sigma_{g,b}$ can be retracted onto the square configuration space of $K_{g,b}(n)$, i.e., we prove Theorem \ref{point is square}.
Section \ref{discrete configuration spaces} contains our proof of Theorem \ref{square is discrete}, showing that the configuration space of unit-squares in $K_{g,0}(n)$ can be retracted to the discrete configuration of $K_{g,0}(n)$, completing the proof of Theorem \ref{point is discrete}.
Finally, in Section \ref{remarks and future directions} we show how our proof Theorem \ref{point is discrete} can be adapted to give an explicit version of Abrams's Theorem 2.1 in \cite{abrams2000configuration}.
We give several suggestions of how discrete Morse theory might be used in conjunction with Theorem \ref{point is discrete} to determine the (co)homology of ordered surface configuration spaces.
Additionally, we suggest how our model might be generalized to other cube complex structures on surfaces and higher dimensional spaces.

\subsection{Acknowledgements}
The author owes a special debt of gratitude to Matthew Kahle and Bradley Zykoski for their suggestions and insights.
Additionally, this work greatly benefited from conversations with Karen Butt, Nir Gadish, Jes\'{u}s Gonz\'{a}lez, Ben Knudsen, and Shmuel Weinberger.
This research was supported in part by an AMS-Simons travel grant.

\section{Cube Complexes}\label{cube complexes}
Our process for discretizing $F_{m}(\Sigma_{g,b})$ begins with finding an appropriate cellular decomposition for the surface $\Sigma_{g,b}$.
To do this, we recall the definition of a special type of cellular complex called a cube complex.
While one can give $\Sigma_{g,b}$ a number of different cube complex structures, we will focus on one family of these structures, which we denote by $K_{g,b}(n)$, whose simple combinatorial and geometric nature will prove to be especially attractive for our purposes.
In particular, we prove that there exists a useful notion of polar coordinates on $K_{g,b}(n)$ in Lemmas \ref{section of small set no boundary} and \ref{section of small set with boundary}; this idea will play a central role throughout the paper.

One of the first constructions one learns when studying algebraic topology is that of a simplicial complex.
Such a complex arises when one glues together triangles (and their higher dimensional analogues) along their sides assuming a few minor restrictions.
Unfortunately, the product of two simplices tends not to be a simplex, making it hard to give a natural combinatorial description of a triangulation of the product of a pair of simplicial complexes. 
If we replace simplices with cubes, taking products becomes much simpler, as a product of two cubes is a cube, and it is easy to name the cells in this product.
We say that a \emph{cubical complex} (or \emph{cube complex}) $K$ is a disjoint union of cubes $X=\bigsqcup_{\lambda\in \Lambda}[0,1]^{k_{\lambda}}$ quotiented by an equivalence relation $\sim$.
The restrictions $p_{\lambda}:[0,1]^{k_{\lambda}}\to K$ of the natural projection map $p:X\to K=X/\sim$ are required to satisfy 
\begin{enumerate}
\item for every $\lambda\in \Lambda$ the map $p_{\lambda}$ is injective, and
\item if $p_{\lambda}\big([0,1]^{k_{\lambda}}\big)\cap p_{\lambda'}\big([0,1]^{k_{\lambda'}}\big)\neq\emptyset$, then there is an isometry $h_{\lambda, \lambda'}$ from a face $T_{\lambda}$ of $[0,1]^{k_{\lambda}}$ to a face $T_{\lambda'}$ of $[0,1]^{k_{\lambda'}}$ such that $p_{\lambda}(x)=p_{\lambda'}(x')$ if and only if $h_{\lambda, \lambda'}(x)=x'$.
\end{enumerate}

In this paper we concern ourselves with certain cube complexes called \emph{square-tiled surfaces}.
These cube complexes are orientable connected surfaces obtained from a finite collection of
unit-squares in $\R^{2}$ after identifications of pairs of parallel sides via translations.
In particular, we are concerned with two families of square-tiled surfaces, which we denote $K_{g,0}(n)$ and $K_{g,1}(n)$, that are homotopy equivalent to the abstract surfaces $\Sigma_{g,0}$ and $\Sigma_{g,1}$, respectively.
\begin{defn}\label{K(n) definition}
Let \emph{$K_{0,1}(n)$} denote the surface homotopy equivalent to $\mathbb{D}^{2}$ that arises from the natural square tiling of the region $[0,n]\times [0,n]$ in $\R^{2}$ by $1\times 1$ squares.

For $g\ge1$ and $b=0$, we define \emph{$K_{g,0}(n)$} to be the surface homotopy equivalent to $\Sigma_{g,0}$ that is obtained by tiling the region $[0, (2g-1)\times(n+1)]\times[0,n+1]$ of $\R^{2}$ with $1\times 1$ squares in the natural way, and identifying the intervals $\{0\}\times[0,n+1]$ and $\{(2g-1)\times(n+1)\}\times [0,n+1]$, as well as the pairs of intervals
$[i(n+1), (m+1)(n+1)]\times \{0\}$ with $[(2g-2-i)(n+1), (2g-1-i)(n+1)]\times\{n+1\}$ for $i\in \{0,1,\dots 2g-2\}$, e.g., Figure \ref{fig:K20(3)}.
We write \emph{$p$} for the point in $K_{g,0}(n)$ that is identified with $\{i(n+1)\}\times \{0\}$ and $\{i(n+1)\}\times \{n+1\}$ for $i\in \{0,\dots, 2g-1\}$.

\begin{figure}[h]
    \centering
    \includegraphics[width=.75\linewidth]{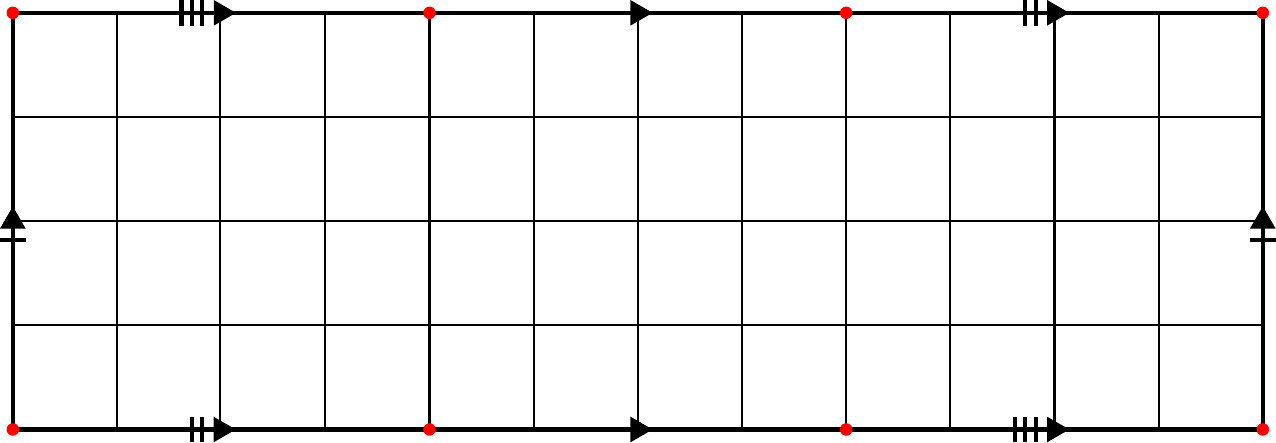}
    \caption{
    The cube complex structure $K_{2,0}(3)$ on $\Sigma_{2,0}$.
    We have highlighted the sole singular point in red. 
    Note that we have a different indexing convention for surfaces with non-zero genus than we do for the plane.
    We do this to ensure that if $m\le n$, the $m^{\text{th}}$-discrete ordered configuration space of $K_{g,0}(n)$ is homotopy equivalent to the $m^{\text{th}}$-ordered configuration space of points in $K_{g,0}(n)$.
    }
    \label{fig:K20(3)}
\end{figure}

Finally, we define \emph{$K_{g, 1}(n)$} to be the square-tiled surface homotopy equivalent to $\Sigma_{g,1}$ that arises from removing the open $l^{\infty}$-balls of radius $\frac{1}{2}$ around the points $\{i(n+1)\}\times \{0\}$ and $\{i(n+1)\}\times \{n+1\}$ for $i\in \{0,\dots, 2g-1\}$ from the rectangle $[0, (2g-1)\times(n+1)]\times[0,n+1]\subset \R^{2}$, and identifying the interval $\{0\}\times[\frac{1}{2},n+\frac{1}{2}]$ with the interval $\{(2g-1)\times(n+1)\}\times [\frac{1}{2},n+\frac{1}{2}]$, as well as the pairs of intervals $[i(n+1)+\frac{1}{2}, (m+1)(n+1)-\frac{1}{2}]\times \{0\}$ and $[(2g-2-i)(n+1)+\frac{1}{2}, (2g-1-i)(n+1)-\frac{1}{2}]\times\{n+1\}$ for $i\in \{0,1,\dots 2g-2\}$, e.g., Figure \ref{fig:K21(3)}.
Note that the unit-squares constituting $K_{g,1}(n)$ are centered at the integer lattice in $\R^{2}$.

\begin{figure}[h]
    \centering
    \includegraphics[width=.75\linewidth]{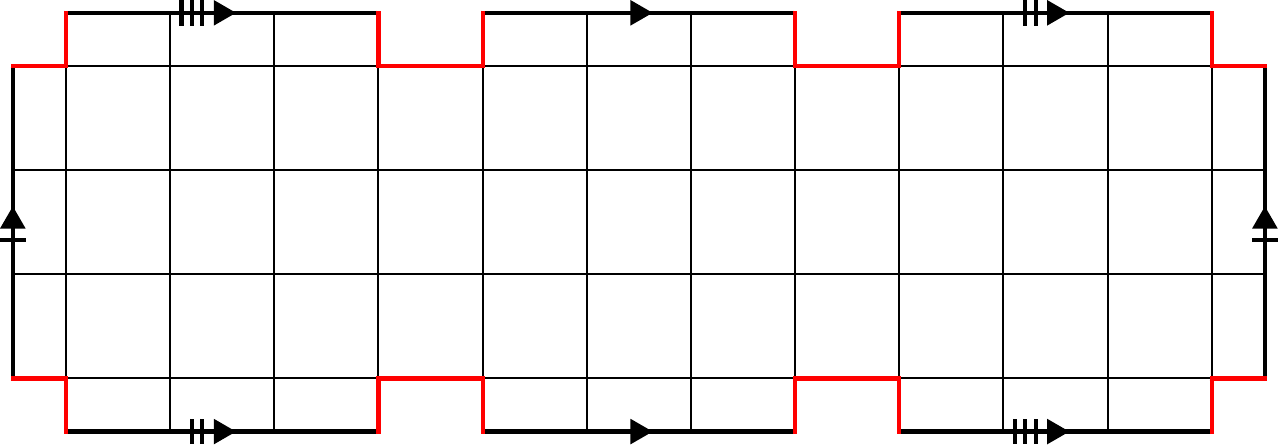}
    \caption{The cube complex structure $K_{2,1}(3)$ on $\Sigma_{2,1}$.
    The boundary of $K_{2,1}(3)$ is colored red.
    Note the thick black lines that are identified are not part of the cube complex structure, though we draw them to emphasize the relation between $K_{g,1}(n)$ and $K_{g,0}(n)$; see Figure \ref{fig:K20(3)}.
    }
    \label{fig:K21(3)}
\end{figure}
\end{defn}

Note that $K_{0,1}(n)$ is the $n\times n$ rectangle studied in \cite{plachta2021configuration, alpert2023homology}.
We will write $K(n)$ for $K_{g,0}(n)$ and $K_{g,1}(n)$ if the genus and number of boundary components are not vital to the situation.

The Euclidean metric on $\R^{2}$ induces a Euclidean metric on $K(n)$ away from $p$, the point of nonzero curvature in $K_{g,0}(n)$; moreover, since these are translation surfaces, there are reasonable notions of up and down, left and right at all other points in $K(n)$.
It follows that, there is a well-defined path metric 
on $K(n)$, and we can consider the vertical and horizontal lengths $l_{vert}$ and $l_{hor}$ of any path $\gamma$.
With this in mind, we get another metric on $K(n)$, which we call the \emph{Chebyshev} or \emph{$l^{\infty}$-metric}, and we write $d_{\infty}$ for the corresponding distance function on $K(n)$.
Given two points $z, z'$ in $K(n)$, we set
\[
d_{\infty}(z, z'):=\inf_{\gamma}\Big\{\max\big\{l_{vert}(\gamma), l_{hor}(\gamma)\big\}\Big\},
\]
where the infimum is taken over all paths $\gamma$ between $z$ and $z'$.
We will be particularly interested in the open ball of radius $R$ centered at $z$ with respect to this $l^{\infty}$-metric, i.e., the set of points $z'$ such that the Chebyshev distance between $z$ and $z'$ is less than $r$:
\[
B_{R}(z):=\{z'|d_{\infty}(z, z')<R\}.
\]
Often we will call these balls open \emph{$2R$-squares} as if $d_{\infty}(z, p)<R$, the ball $B_{R}(z)$ is isometric to a Euclidean square of side length $2R$.
If $z$ is Chebyshev distance less than $R$ from $p$ and $R$ is relatively small, then these balls look like stars, see Figure \ref{fig:ballsinK20(3)}, though we will still call them squares.
\begin{figure}[h]
    \centering
    \includegraphics[width=.75\linewidth]{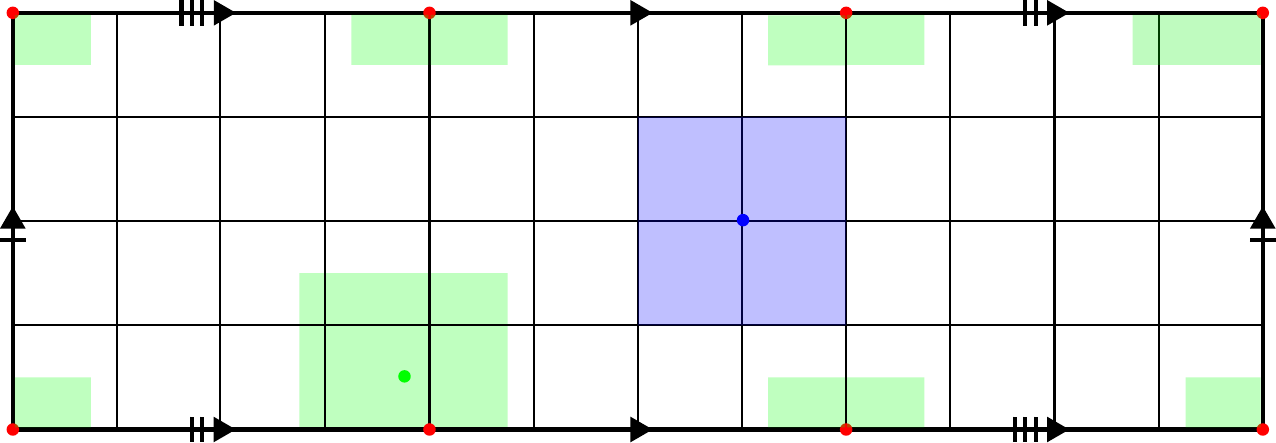}
    \caption{Two $l^{\infty}$-balls of diameter $2$ in $K_{2,0}(3)$.
    The blue ball looks like a square, whereas the green ball looks like a $12$-pointed star as it contains $p$, the point of non-zero curvature highlighted in red. 
    }
    \label{fig:ballsinK20(3)}
\end{figure}

Note that for $g\ge 1$, the surface $K_{g,1}(n)$ can be obtained from  $K_{g,0}(n)$ by removing the open $l^{\infty}$-ball of radius $\frac{1}{2}$ centered at $p$ and shifting the cube structure.
Furthermore, if $g\ge 1$, then all the curvature of $K_{g,0}(n)$ is concentrated at $p$, which has angle $(2g-1)2\pi$; every other point of $K_{g,0}(n)$, hence $K_{g,1}(n)$, has zero curvature.
With respect to either the Euclidean or Chebyshev metric, the injectivity radius of $K_{g,0}(n)$ is $\frac{n+1}{2}$, a fact that will play a small role in our proofs of Theorems \ref{point is square} and \ref{square is discrete}.

\begin{prop}\label{injectivity radius}
The injectivity radius of $K_{g,0}(n)$ is $\frac{n+1}{2}$.
\end{prop}

Notably, it follows from this observation that any loop in $K_{g,0}(n)$ of Chebyshev length at most $n$ is contractible; in $K_{g,1}(n)$ any such loop is homotopic into the boundary.
This, combined with the observations that at $p\in K_{g,0}(n)$ there is a cone point of angle $(2g-1)2\pi$, lends itself to describing small subspaces of  $K(n)$ in terms of a form of polar coordinates, where the origin corresponds to $p$. 
We formalize this idea by defining a space $\tilde{K}(n)$ that surjects onto $K(n)$.

\begin{defn}
Let \emph{$S(n)$} denote the $(2n+2)\times(2n+2)$ closed square centered at the origin in $\R^{2}$, i.e., $S(n):=[-n-1, n+1]\times [-n-1, n+1]\subset \R^{2}$.
Cut $S(n)$ along the line $(0, n+1]\times\{0\}$.
We define \emph{$\tilde{K}_{g,0}(n)$} by gluing together $2g-1$ copies of this cut $S(n)$ by identifying the bottom side of slit of the $i^{\text{th}}$-copy of $S(n)$ with the top side of the slit of the $(i+1)^{\text{th}}$-copy of $S(n)$ for $i\in \Z/(n+1)\Z$---this identifies all of the origins, and we will call the resulting point the origin of $\tilde{K}_{g,0}$.
There is a well-defined projection map $\pi:\tilde{K}_{g,0}(n)\to K_{g,0}(n)$ that arises from identifying the top side of the slit along $(0,n+1]\times \{0\}$ in the first $S(n)$ with the interval $(n+1, 2n+2]\times\{0\}$ in $K_{g,0}(n)$ in an orientation preserving locally isometric manner; see Figure \ref{fig:projectionmaptosurface}.

Similarly, we define \emph{$A(n)$} to be the annulus in $\R^{2}$ given by $S(n)\backslash B(n)$, where $B(n)$ is the set of open $l^{\infty}$-balls of radius $\frac{1}{2}$ centered at the $\big((n+1)\Z\big)^{2}$ lattice points in $\R^{2}$; see Figure \ref{fig:projectionmaptosurface}.
We define \emph{$\tilde{K}_{g,1}(n)$} to the subspace of $\tilde{K}_{g,0}(n)$ corresponding to the $A(n)$s, yielding a map, which we also denote by $\pi$, from $\tilde{K}_{g,1}(n)$ to $K_{g,1}(n)$.
\end{defn}

At most points of $\tilde{K}_{g,0}(n)$ the map $\pi$ is $4$-to-$1$; however, there are a few exceptions.
Notably, along the intervals $\big(i(n+1), (n+1)(i+1)\big)\times\{0\}$ and $\big\{i(n+1)\big\}\times\big(0, n+1\big)$ where $i\in \{0, 2g-2\}$ the map is $(4g-2)$-to-$1$; moreover, at $p$, the point of potentially non-zero curvature, the map is $(16g-7)$-to-$1$; see Figure \ref{fig:projectionmaptosurface}

Next, we prove that there is a section of $\pi$ on any ``small'' set $U\subset K(n)$, i.e., reasonably small subspaces of $K(n)$ can be described in terms of	 ``polar coordinates.''
In Section \ref{square configuration spaces} we use this fact to find a way of moving points in a configuration farther apart, and in Section \ref{discrete configuration spaces} we use this fact to retract $SF_{m}\big(K_{g,0}(n)\big)$ onto $DF_{m}\big(K_{g,0}(n)\big)$.

\begin{figure}[H]
    \centering
\includegraphics[height=.8\textheight,keepaspectratio]{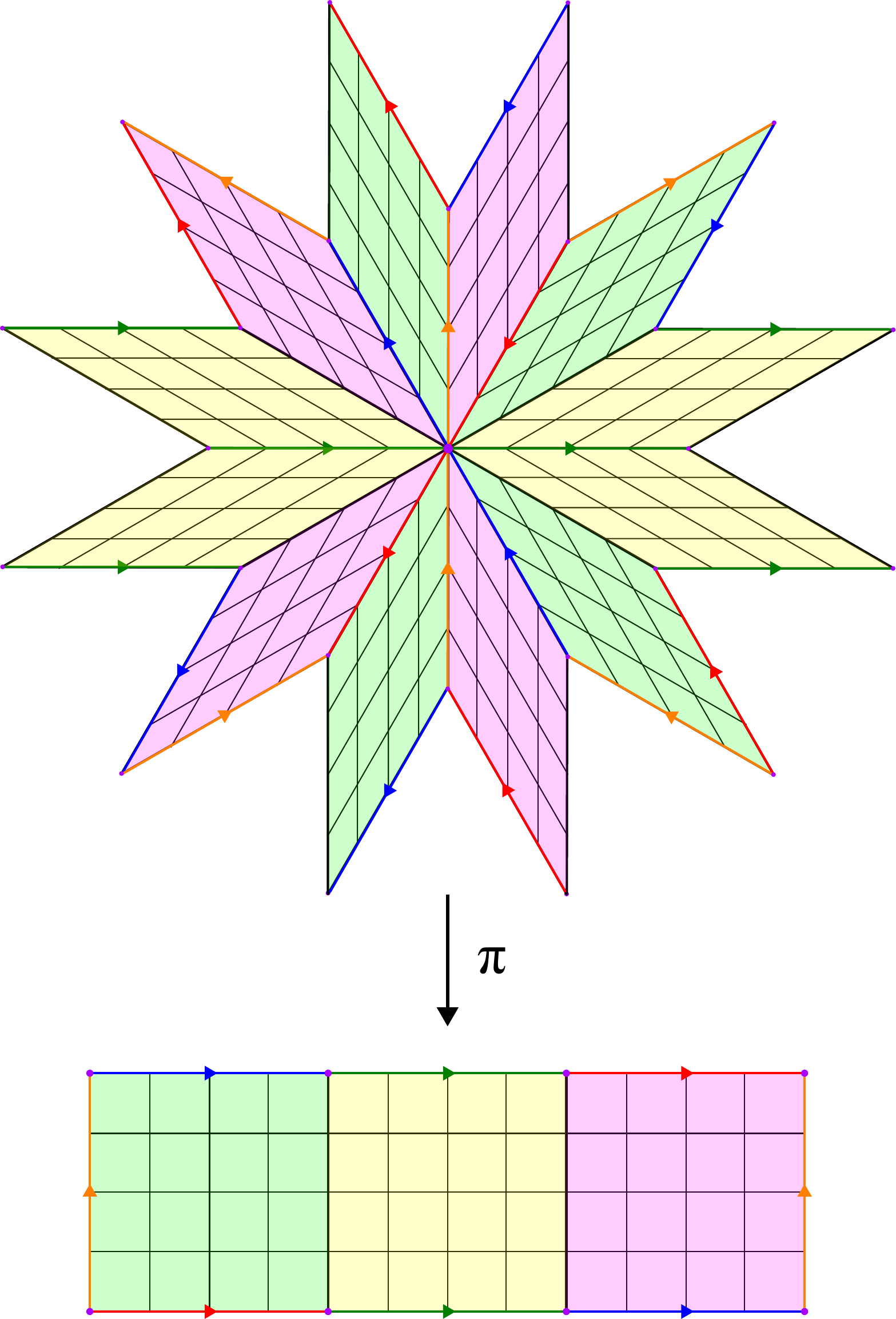}
    \caption{The projection map $\pi$ from $\tilde{K}_{2,0}(3)$ to $K_{2,0}(3)$.
    Note that we have colored the various cells of $\tilde{K}_{2,0}(3)$ and $K_{2,0}(3)$ so that $\pi$  preserves the coloring.
    Moreover, the $1$-cells in $K_{2,0}(3)$ that are colored (a color other than black) the same are identified.
    }
    \label{fig:projectionmaptosurface}
\end{figure}

\begin{lem}\label{section of small set no boundary}
Let $U$ be a path-connected contractible subspace of $K_{g,0}(n)$ of diameter at most $n$ in the intrinsic Chebyshev metric.
Then, there is a continuous section $s:U\to \tilde{K}_{g,0}(n)$ of the projection map $\pi:\tilde{K}_{g,0}(n)\to K_{g,0}(n)$, such that $s(U)$ is contained in the closed $l^{\infty}$-ball of radius $n+\frac{1}{2}$ about the origin in $\tilde{K}_{g,0}(n)$.
\end{lem}

\begin{proof}
If $U$ contains $p$, then the map $s$ that sends $p$ to the origin in $\tilde{K}_{g,0}(n)$ extends to a section $s:U\to \tilde{K}_{g,0}(n)$.
To see this, let $\gamma(t)$ be a path in $U\subset K_{g,0}(n)$ starting at $p$, and consider the straight-line $L_{\gamma(t)}$ in $K_{g,0}(n)$ from $p$ to $\gamma(t)$ that arises from starting with the trivial line at $p$ and moving along $\gamma$.
If $l_{\gamma(t)}$ denotes the Euclidean length of $L_{\gamma(t)}$ and $\theta_{\gamma(t)}$ denotes the angle $L_{\gamma(t)}$ makes with the line segment $[n+1, 2n+2)\times\{0\}$ at $\{n+1\}\times \{0\}=p$, we send $\gamma(t)$ to the point $\tilde{\gamma}(t)$ in $\tilde{K}_{g,0}(n)$ that is distance $l_{\gamma(t)}$ from the origin and is such that the straight-line in $\tilde{K}_{g,0}(n)$ through the origin and $\tilde{\gamma}(t)$ and the top side of the interval $[0,n+1]$ in the first copy of $S(n)$ in $\tilde{K}_{g,0}(n)$ make an angle of $\theta_{\gamma(t)}$ at the origin.
We claim that this map $s:\gamma(t)\mapsto\tilde{\gamma}(t)$ is a continuous section of $\pi:\tilde{K}_{g,0}(n)\to K_{g,0}(n)$.

First, we check that $s$ is well-defined.
It suffices to show that if $\gamma(t)$ and $\gamma'(t)$ are two different paths $[0,1]\to U$ from $p$ to $q$, then they lead to the same point $\tilde{q}$ in $\tilde{K}_{g,0}(n)$, i.e., $\tilde{\gamma}(1)=\tilde{\gamma}(1)$.
Assume they do not.
In this case the loop $\gamma^{-1}\gamma'$ in $K_{g,0}(n)$ is non-contractible, as such a concatenation of paths would yield a path starting at one of the points $\big(i(n+1), 0\big)$ or $\big(i(n+1),n+1\big)$ in  $\big[0,(2g-1)(n+1)\big]\times[0,n+1]$ identified with $p$ and ending at a different one of these points.
Any such path is a non-contractible loop; see Figure \ref{fig:longpathnopolar} for an example.
This contradicts our assumptions on the contractibility of $U$, so $s$ must be well-defined.

\begin{figure}[H]
    \centering
\includegraphics[width=.8\textwidth,keepaspectratio]{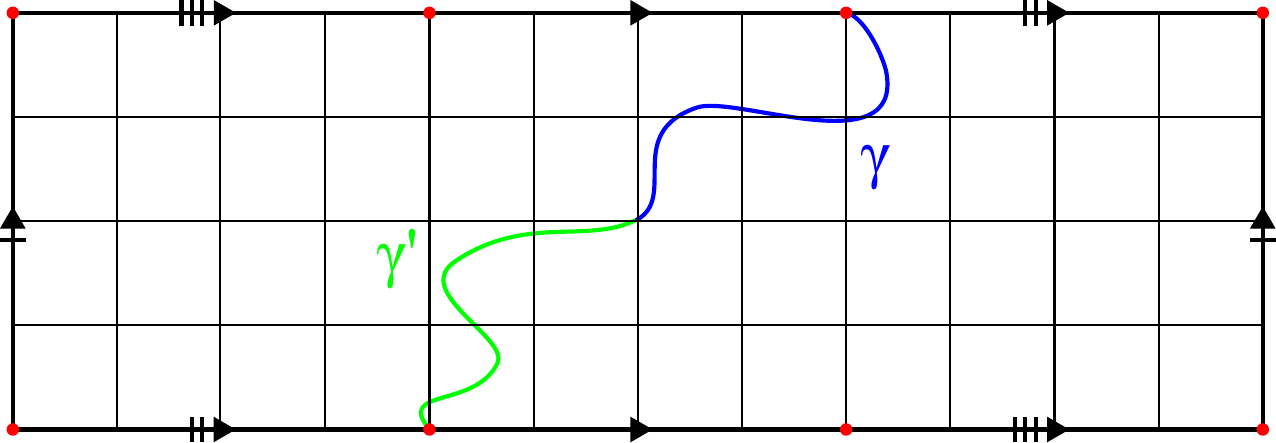}
    \caption{A non-contractible loop in $K_{2,0}(3)$ arising from the concatenation $\gamma^{-1}\gamma'$ of two paths $\gamma$ and $\gamma'$ starting at $p$.
    If $\gamma'$ is the green path and $\gamma$ is the blue path, we have that $\theta\big(\gamma'(1)\big)=\frac{\pi}{4}$ and $\theta\big(\gamma(1)\big)=\frac{13\pi}{4}$ despite $\gamma(1)=\gamma'(1)$.
    Our restrictions on $U$ allow us to avoid such a situation and the resulting concerns about the well-definedness of $s$.
    }
    \label{fig:longpathnopolar}
\end{figure}

Additionally, $s$ is continuous, and its image must be contained in $l^{\infty}$-ball of radius $n$ around the origin in $\tilde{K}_{g,0}(n)$ by our assumptions on the diameter of $U$.

If $U$ does not contain $p$, then there is some ray of Chebyshev length at most $n$ that starts at $p$ ands hits $U$.
If $U$ contains two of the segments of the form $\big(i(n+1), (n+1)(i+1)\big)\times\{0\}$ or $\big\{i(n+1)\big\}\times\big(0, n+1\big)$ for $i\in \{0, 2g-2\}$, then there must be a path $\gamma$ in $U$ that connects these two edges.
Moreover, we may take these two edges to be two adjacent sides of one of the squares of the form $\big(i(n+1), (n+1)(i+1)\big)\times(0,n+1)$ for $i\in \{0, 2g-2\}$.
To see this, note that a path in contained in one of these squares cannot have end-points on opposite sides of the square as such a path would be of length greater than $n$, too long for our restrictions on $U$.
Take the ray from $p$ that starts at the corner of the square where the sides meet and goes along one of the edges of the square.
Lifting this ray to $\tilde{K}_{g,0}$ via the angle and length method defined above yields a well-defined continuous section of $\pi:\tilde{K}_{g,0}(n)\to K_{g,0}(n)$ on $U$.
Moreover, the image of this section must be contained in the $l^{\infty}$-ball of radius $n$ about the origin in $\tilde{K}_{g,0}(n)$, as if it does not, then the diameter of $U$ would be too large.
If $U$ only intersects one of the edges of the form $\big(i(n+1), (n+1)(i+1)\big)\times\{0\}$ or $\big\{i(n+1)\big\}\times\big(0, n+1\big)$ for $i\in \{0, 2g-2\}$, then one can take either ray starting at one of the end points of this edge going along it, yielding a lift. 
Moreover, one of these choices guarantees the image of $U$ is contained in the $l^{\infty}$-ball of radius $n+\frac{1}{2}$ about the origin in $\tilde{K}_{g,0}(n)$, as otherwise the diameter of $U$ would too large.
Finally, if $U$ does not contain any of these edges, it is entirely contained in one of the squares of the form $\big(i(n+1), (n+1)(i+1)\big)\times(0,n+1)$ for $i\in \{0, 2g-2\}$.
Any ray entirely contained in this square starting at $p$ gives a section, though we may choose our section to ensure its image lies in the desired subspace of $\tilde{K}_{g,0}(n)$.
\end{proof}

This section of $\pi:\tilde{K}_{g,0}(n)\to K_{g,0}(n)$ yields a section of $\pi:\tilde{K}_{g,1}(n)\to K_{g,1}(n)$ on small subspaces of $K_{g,1}(n)$.

\begin{lem}\label{section of small set with boundary}
Let $U$ be a connected subspace of $K_{g,1}(n)$ such that the diameter of $U$ is at most $n$.
Then there is a section $s:U\to \tilde{K}_{g,1}(n)$ of the projection map $\pi:\tilde{K}_{g,1}\to K_{g,1}(n)$.
\end{lem}

\begin{proof}
By construction, $K_{g,1}(n)$ can be viewed as a subspace of $K_{g,0}(n)$; moreover, the preimage of $K_{g,1}(n)$ under $\pi$ is $\tilde{K}_{g,1}(n)$, which in turn is a subspace of $\tilde{K}_{g, 0}(n)$.
Therefore, we can view $U$ as a subspace of $K_{g,0}(n)$, and consider the section $s$ of Lemma \ref{section of small set no boundary} on $U$. 
Since the image of this section lies in the preimage of $K_{g,1}(n)$, that is $\tilde{K}_{g,1}(n)$, this section restricts to a section $s$ of $\pi:\tilde{K}_{g,1}(n)\to K_{g,1}(n)$ on $U$.
\end{proof}

Note that the sections $s$ are distance preserving (in the inherit metric) on $U$, a fact that will come in handy in Sections \ref{square configuration spaces} and \ref{discrete configuration spaces}.

In Section \ref{discrete configuration spaces} we will need the notion of a dual complex due to the boundary of $K_{g,1}(n)$, so we recall its definition here.
Given a cube complex $K$ coming from a square-tiled surface, there is a natural \emph{dual complex}, which we denote by $K^{*}$.
For each $2$-cell of $K$, there is a $0$-cell of $K^{*}$, and we connect two $0$-cells of $K^{*}$ with a $1$-cell if the corresponding $2$-cells of $K$ share a $1$-dimensional face.
Additionally, there is a $2$-cell in $K^{*}$ for every $0$-cell in the interior of $K$.
The boundary of this $2$-cell consists of the $0$- and $1$-cells of $K^{*}$ dual to the cells of $K$ containing the corresponding $0$-cell.
In general, this dual complex is not a cube complex; however, if $K$ is of the form $K_{1,0}(n)$ or $K_{g,1}(n)$, then the dual complex $K^{*}$ is a cube complex.
For $g\ge 1$, the complex $K^{*}_{g,1}(n)$ arises from $K_{g,0}(n)$ by removing the open $l^{\infty}$-ball of radius $1$ around $p$, that is, $K^{*}_{g,1}(n)$ is the subcomplex of $K_{g,0}(n)$ consisting of all cells whose closure does not contain $p$; see Figure \ref{fig:K21(3)dual} for an example.
We will use this dual complex to construct a discrete configuration space homotopy equivalent to $F_{m}(\Sigma_{g,1})$.

\begin{figure}[h]
    \centering
    \includegraphics[width=.75\linewidth]{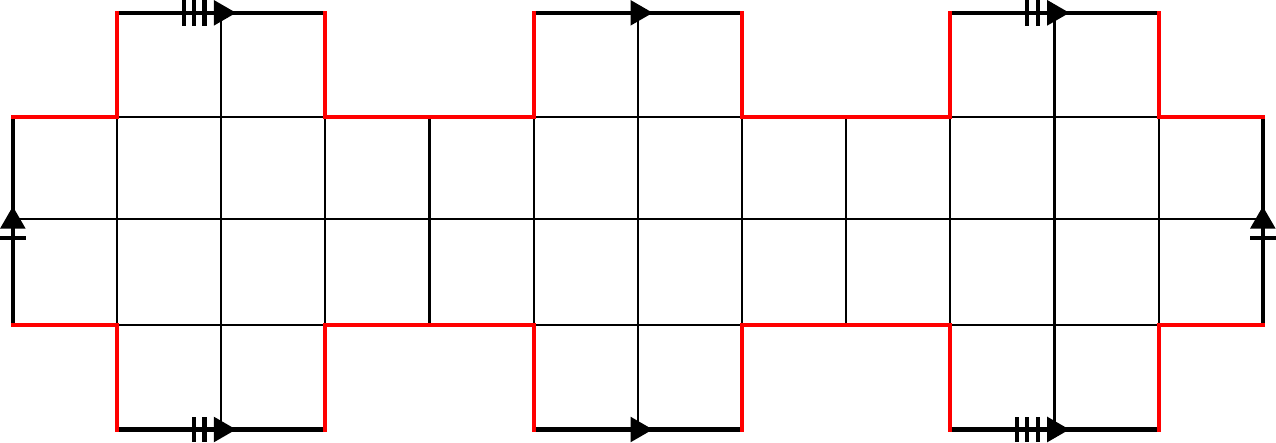}
    \caption{
    The cube complex $K^{*}_{2,1}(3)$, with boundary in red.
    Note that $K^{*}_{g,1}(n)$ is a subcomplex of $K_{g,0}(n)$, a fact we will use in Section \ref{discrete configuration spaces} to prove that if $m\le n$, the $m^{\text{th}}$-discrete ordered configuration of $K^{*}_{g,1}(n)$ is homotopy equivalent to $F_{m}(\Sigma_{g,1})$.
    }
    \label{fig:K21(3)dual}
\end{figure}

In the next two sections we will use the complexes $K_{g,b}(n)$ and $K^{*}_{g,b}(n)$ to find a discrete model for the ordered configuration space of points in a surface (with boundary).
We begin by noting that the cube complex structure of $K_{g,b}(n)$ lends itself to the notion of the configuration space of squares in $K_{g,b}(n)$.
We will show that as long as the squares are small enough, this square configuration space is homotopy equivalent to point configuration space.

\section{Square Configuration Spaces}\label{square configuration spaces}

Our proof that there is a simple discrete model for the configuration space of points in a surface factors through square configuration spaces.
In this section, we consider the Chebyshev metric on our model $K_{g,b}(n)$ of the surface $\Sigma_{g,b}$, and see that there is a notion of the ordered configuration space of squares in $\Sigma_{g,b}$. 
Further, we prove Theorem \ref{point is square}, showing that if the squares are small enough, then the configuration space of $m\le n$ squares in $K_{g,b}(n)$ is homotopy equivalent to the configuration space of $m$ points in $\Sigma_{g,b}$.
This equivalence and that of Theorem \ref{square is discrete} combine to yield Theorem \ref{point is discrete}.

In order to prove that the configuration space of points in a surface $\Sigma_{g,b}$ is homotopy equivalent to the configuration space of squares in the homeomorphic cube complex, we must define square configuration space.
In the previous section, we noted that there is a Chebyshev or $l^{\infty}$-metric on $K_{g,b}(n)$ arising from the Euclidean metric on $\R^{2}$ and that there well-defined notions of left and right, up and down at every point other than the sole singular point $p$ in $K_{g,0}(n)$.
This metric is flat everywhere in $K_{g,1}(n)$ and everywhere in $K_{g,0}(n)$ other than $p$, which is a cone point of angle $2\pi(2g-1)$.
In particular, this metric leads to the notion of an open $l^{\infty}$-ball of diameter $d$ centered at $z\in K$, being the set of all points $z'$ whose horizontal and vertical distances to $z$ is at most $\frac{d}{2}$, i.e., the set of points $z'$ such that $d_{\infty}(z', z')<\frac{d}{2}$.
If this open $l^{\infty}$-ball does not intersect the boundary of $K$, then we call it an \emph{open $d$-square}.
Note that if our surface $K(n)$ is of the form $K_{g,1}(n)$ or $K_{1,0}(n)$ and $d$ is less than $\frac{n+1}{2}$, then the open $d$-squares look like squares, i.e., consist of four segments of length $d$ joined at right-angles.
This is also true at most points $z$ of $K_{g,0}(n)$ for $g>1$ as well, though if $d_{\infty}(z, p)<\frac{d}{2}$, the open $l^{\infty}$-ball of diameter $d$ centered at $z$ is distorted and becomes star-shaped; see Figure \ref{fig:ballsinK20(3)}.
Despite these deformities, we call such balls squares.

Given a square-tiled surface $K(n)$ of the form $K_{g,0}(n)$ or $K_{g,1}(n)$ we write $SF_{m}\big(K(n), d\big)$ for the \emph{$m^{\text{th}}$-ordered configuration space of open $d$-squares in $K(n)$}.
This is the subspace of points in $\big(K(n)\big)^{m}$ whose coordinates are Chebyshev distance at least $d$ apart and Chebyshev distance at least $\frac{d}{2}$ from the boundary of $K(n)$, i.e., 
\[
SF_{m}\big(K(n), d\big):=\left\{(z_{1},\dots, z_{m})\in \big(K(n)\big)^{m}|d_{\infty}(z_{i}, z_{j})\ge d\;\forall i\neq j, d_{\infty}\big(z_{i},\partial K(n)\big)\ge \frac{d}{2}\;\forall i\right\};
\]
see Figure \ref{fig:pointinSF3K21(3)} for an example of a configuration in a square configuration space.
Note that the symmetric group $S_{m}$ acts on $SF_{m}\big(K(n),d\big)$ by permuting the labels of the points.
Taking the quotient of $SF_{m}\big(K(n),d\big)$ by this action yields $SC_{m}\big(K(n), d\big)$ the \emph{$m^{\text{th}}$-unordered configuration space of $d$-squares in $K(n)$}.
Though we only work with ordered square configuration spaces, our results also hold for unordered square configuration spaces.
When $d=1$, we write $SF_{m}\big(K(n)\big)$ for $SF_{m}\big(K(n),1\big)$.

\begin{figure}[h]
    \centering
    \includegraphics[width=.75\linewidth]{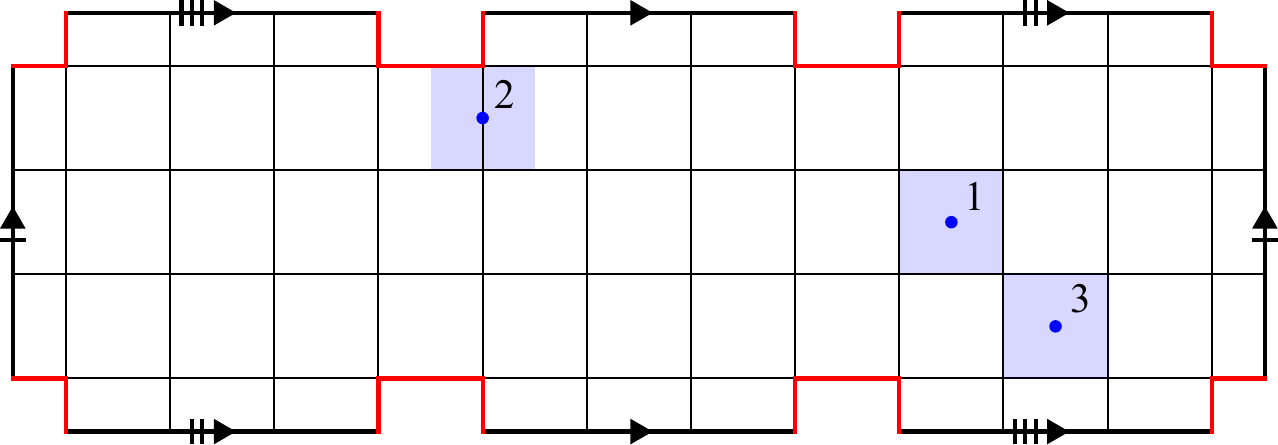}
    \caption{
    A configuration in $SF_{3}\big(K_{2,1}(3)\big)$.
    Since $d_{\infty}\big(z_{2},\partial K_{2,1}(3)\big)=\frac{1}{2}$ and $d_{\infty}(z_{1},z_{3})=1$, this configuration does not belong in $SF_{3}\big(K_{2,1}(3),d\big)$ for $d>1$.
    }
    \label{fig:pointinSF3K21(3)}
\end{figure}

We will use affine Morse--Bott theory to prove that if $m\le n$, then $F_{m}\big(K(n)\big)$ and $SF_{m}\big(K(n),d\big)$ are homotopy equivalent for $d\le1$.
To use these techniques we generalize Plachta's definition of the tautological function $\theta_{m}$ on $\big(K_{0,1}(n)\big)^{m}$ \cite{plachta2021configuration} to general $\big(K(n)\big)^{m}$.
We will show that these functions are affine Morse--Bott, and prove that their smallest critical values are always at least $\frac{1}{2}$.
This will allow us to use a standard discrete Morse theoretical argument to prove Theorem \ref{point is square}.

We begin by recalling the definitions of affine polytope complexes and affine Morse--Bott functions.
A CW-complex $X$ of dimension $l\ge 1$ is an \emph{affine polytope complex of dimension $l$} if each closed cell $c$ of $X$  is equipped with a characteristic function $\chi_{c}:p_{c}\to X$ such that
\begin{enumerate}
\item $p_{c}$ is a convex polyhedral cell in $\R^{l}$ for each $c$;
\item $\chi_{c}$ is an embedding for each closed cell $c$;
\item the restriction of $\chi_{c}$ to any face of $p_{c}$ coincides with the characteristic function of another cell precomposed with an affine homeomorphism of $\R^{l}$.
\end{enumerate}

It is immediate from their definitions that the complexes $K_{g,0}(n)$ and $K_{g,1}(n)$ are affine polytope complexes of dimension $2$.
This implies that $(K_{g,0}(n))^{m}$ and $(K_{g,1}(n))^{m}$ are also affine polytope complexes.

Given $d>0$ and a point $\textbf{z}=(z_{1}, \dots, z_{m})\in \big(K(n)\big)^{m}$, where $K(n)$ is of the form $K_{g,0}(n)$ or $K_{g,1}(n)$, we would like to determine whether $\textbf{z}$ is also in $SF_{m}\big(K(n), d\big)$.
We do this by considering the \emph{particular metric functions}
\[
d_{i,j}(\textbf{z}):=d_{\infty}(z_{i}, z_{j})
\]
and 
\[
d_{i}(\textbf{z}):=d_{\infty}\big(z_{i}, \partial K_{g,1}(n)\big),
\]
which measure the distance between pairs of points in a configuration and the distance from a point in a configuration to the boundary of $K_{g,1}(n)$, respectively.
We use these functions to define the \emph{tautological function} $\theta_{m}$, which we use to determine whether or not a configuration $\textbf{z}$ is in $SF_{m}\big(K(n), d\big)$:
\[
\theta_{m}(\textbf{z}):=\min\left\{\frac{1}{2}\min_{i\neq j}d_{\infty}(z_{i}, z_{j}),\min_{k}d_{\infty}\big(z_{k}, \partial K(n)\big)\right\}
\]
Namely, $\textbf{z}\in SF_{m}\big(K(n),d\big)$ if and only if $d\le 2\theta_{m}(\textbf{z})$; as a result, we will devote considerable time and energy to studying this function.
In order to use $\theta_{m}$ to prove the homotopy equivalence of $F_{m}\big(\Sigma_{g,b}\big)$ and $SF_{m}\big(K_{g,b}(n)\big)$ for $m\le n$, we need to show that $\theta_{m}$ is an \emph{affine Morse--Bott function}, i.e., a continuous map $f$ from an affine polytope complex $X$ to $\R$ such that $f\circ \chi_{c}:p_{c}\to \R$ extends to an affine function $\R^{l}\to \R$ for each cell $c\in X$.
Then, we must determine its critical values.
The first step in this process is to prove that $d_{\infty}(z_{1}, z_{2})$ is an affine Morse--Bott function on $\big(K(n)\big)^{2}$.

\begin{prop}\label{distanct function between two points is affine Morse--Bott}
If $K(n)$ is a cube complex of the form $K_{g,0}(n)$ or $K_{g,1}(n)$, the Chebyshev distance function $d_{1,2}:\big(K(n)\big)^{2}\to \R$ is an affine Morse--Bott with respect to some subdivision $Q_{2}$ of the complex $\big(K(n)\big)^{2}$.
\end{prop}

\begin{proof}
There is a natural polytopal structure on $K_{g,0}(n)$ arising from the integer grid in $\R^{2}$---for $K_{g,1}(n)$, the offset integer grid---however, this polytopal structure is not fine enough.
We remedy this by subdividing products of closed $2$-cells in $K(n)$, and showing that this leads to an affine polytopal structure on $K(n)$.
Namely, given a configuration $\textbf{z}=(z_{1},z_{2})\in C_{1}\times C_{2}\subset\big(K(n)\big)^{2}$, where $C_{1}$ and $C_{2}$ are closed $2$-cells, consider the two parameter family $B_{r}(z_{1})$ of $l^{\infty}$-balls centered at points $z_{1}\in C_{1}\subset K(n)$.
Translating $z_{1}$ leads to translations of the sides of a ball, and increasing the radius from $r$ to $r+r'$ translates the sides of a ball a distance of $r'$.
Note that for points on the horizontal, resp. vertical, sides of the boundary of $B_{r}(z_{1})$ in $K(n)$ the distance to $z_{1}$ is determined via a single vertical, resp. horizontal, measurement. 
Moreover, which of the horizontal, resp. vertical, sides of the $l^{\infty}$-ball centered at $z_{1}$ in $C_{1}$ first touches $z_{2}$ in $C_{2}$ determines how one should go about measuring that vertical, resp. horizontal, distance from $z_{1}$ to $z_{2}$.
Additionally, the intersection of any two sides of the $l^{\infty}$-ball centered at $z_{1}\in C_{1}$ form a hyperplane in $C_{2}$, so varying $z_{1}$ in $C_{1}$ and $r$ yields a hyperplane in $C_{1}\times C_{2}$.
Thus, we see that these hyperplanes in $C_{1}\times C_{2}$ arising from the intersection of the sides $l^{\infty}$-balls determine which of the corresponding two methods of determining distance yields a smaller measurement, and on the hyperplane these measurements agree.
Since these distance functions are affine, it follows that the various methods of determining distance are affine functions on the corresponding subspace of $C_{1}\times C_{2}$.
Since there can only be finitely many sides of an $l^{\infty}$-ball in $K(n)$, we see that $C_{1}\times C_{2}$ is subdivided by a finite number of hyperplanes into regions on which $d_{\infty}(z_{1}, z_{2})$ is an affine function. 
Further subdividing via hyperplanes to ensure that the regions are convex polytopes yields the desired subdivision that makes $d_{\infty}(z_{1}, z_{2})$ an affine Morse--Bott function on each $C_{1}\times C_{2}$, implying that if we take the union of these subdivisions over all such products, then  $d_{1,2}(z_{1}, z_{2})$ an affine Morse--Bott function on $\big(K(n)\big)^{2}$.
\end{proof}

It follows that $\theta_{2}$ is an affine Morse--Bott function on $\big(K_{g,0}(n)\big)^{2}$, the the resulting subdivision might seem a bit mysterious.
We illustrate this subdivision for the torus $K_{1,0}(n)$, showing that the resulting polytopal complex is relatively simple.

\begin{exam}
Given a point $z_{1}\in K_{1,0}(n)$ consider the $l^{\infty}$-ball of radius $r\le\frac{n+1}{2}$ centered at $z_{1}$.
Such a ball has four sides, and the distance to any point on a side (other than the corners) is determined via a single measurement, e.g., the distance from $z_{1}$ to a point $z_{2}$ on the top side of the ball is measured by the vertical length of the straight-line path that starts at $z_{1}$ and goes up to $z_{2}$.
Thus, the horizontal, resp. vertical, sides of the ball yield hyperplanes $y_{1}=y_{2}$, resp. $x_{1}=x_{2}$ that tell us whether distance should be measured by going up or down, resp. left or right.
At the corners of such a ball the corresponding vertical and horizontal methods of measuring distance yield the same result.
Varying $z_{1}$ and $r$ we see that the top left and bottom right corners yield the hyperplane $x_{1}-y_{1}=y_{2}-x_{2}$ in $K_{1,0}(n)$, whereas the top right and bottom left corners yield the hyperplane $x_{1}-y_{1}=x_{2}-y_{2}$.
When $r=\frac{n+1}{2}$ the top and bottom sides of the $l^{\infty}$-ball overlap and the left and right sides do as well. 
This corresponds to the hyperplanes $y_{1}=y_{2}+\frac{n+1}{2}$ and $x_{1}=x_{2}+\frac{n+1}{2}$, respectively.
These six hyperplanes divide $\big(K_{1,0}(n)\big)^{2}$ into eight polytopal regions on the interior of which $d_{\infty}(z_{1}, z_{2})$ is defined via a single horizontal or vertical measurement.
Moreover, they determine whether distance should be measured via moving from $z_{1}$ to $z_{2}$ by going up or down, right or left; see Figure \ref{fig:K10(3)hyperplanes}.
Taking the common refinement of the resulting subdivision with the natural subdivision of $K_{1,0}(n)$ by the unit-squares ensures that $d_{\infty}(z_{1}, z_{2})$ is affine Morse--Bott.

\begin{figure}[h]
    \centering
    \includegraphics[width=.5\linewidth]{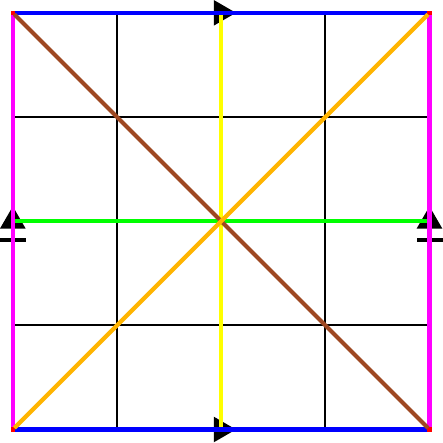}
    \caption{
    The six hyperplanes that divide $\big(K_{1,0}(3)\big)^{2}$ that make $d_{\infty}(z_{1}, z_{2})$ an affine Morse--Bott function.
    We have fixed $z_{1}$ at the ``center'' of $K_{1,0}(3)$.
    The green line is the hyperplane $y_{1}=y_{2}$, the blue line is the hyperplane $y_{1}=y_{2}+2$, the yellow line is the hyperplane $x_{1}=x_{2}$, the pink line is the hyperplane $x_{1}=x_{2}+2$, the orange line is the hyperplane $x_{1}-y_{1}=x_{2}-y_{2}$, and the brown line is the hyperplane $x_{1}-y_{1}=y_{2}-x_{2}$.
    Note that these hyperplanes divide $K_{1,0}(3)$ into $8$-regions, on each of which $d_{\infty}(z_{1}, z_{2})$ is given by a single horizontal or vertical distance measurement.
    }
    \label{fig:K10(3)hyperplanes}
\end{figure}
\end{exam}

In the next proposition we confirm that $\theta_{2}$ is affine Morse--Bott on $\big(K_{g,1}(n)\big)^{2}$.

\begin{prop}\label{tautological function is affine Morse--Bott on 2 points}
The tautological function $\theta_{2}:\big(K_{g,1}(n)\big)^{2}\to \R$ is affine Morse--Bott with respect to some subdivision $Q$ of the complex $\big(K_{g,1}(n)\big)^{2}$.
\end{prop}

\begin{proof}
Recall that we get a polytopal structure on $K_{g,1}(n)$ by identifying the sides of a polygonal region in $\R^{2}$ and including the $\frac{1}{2}$ offset integer grid in $\R^{2}$; see Definition \ref{K(n) definition}.
We further refine this subdivision of $K_{g,1}(n)$ by including the horizontal lines $y=0$ and $y=\frac{n+1}{2}$, the vertical lines $x=i\Big(\frac{n+1}{2}\Big)$ for $i\in \{0, \dots, 4g-3\}$, and the diagonal lines $y=x-i(n+1)$ and $y=-x+i(n+1)+n+1$ for $i\in \{0,\dots, 2g-2\}$; see Figure \ref{fig:K21(3)firstrefinement}.
Call this subdivision $D_{g,1}(n)$.
This subdivision corresponds to which side of an $l^{\infty}$-ball around a point $z\in K_{g,1}(n)$ first touches $\partial K_{g,1}(n)$.
Given this structure, the function $d_{\infty}\big(z, \partial K_{g,1}(n)\big)$ is affine Morse--Bott.
We will show that the common refinement of this subdivision and that of Proposition \ref{distanct function between two points is affine Morse--Bott} can be further refined to give $\theta_{2}$ the structure of an affine Morse--Bott on $\big(K_{g,1}(n)\big)^{2}$

\begin{figure}[h]
    \centering
    \includegraphics[width=.75\linewidth]{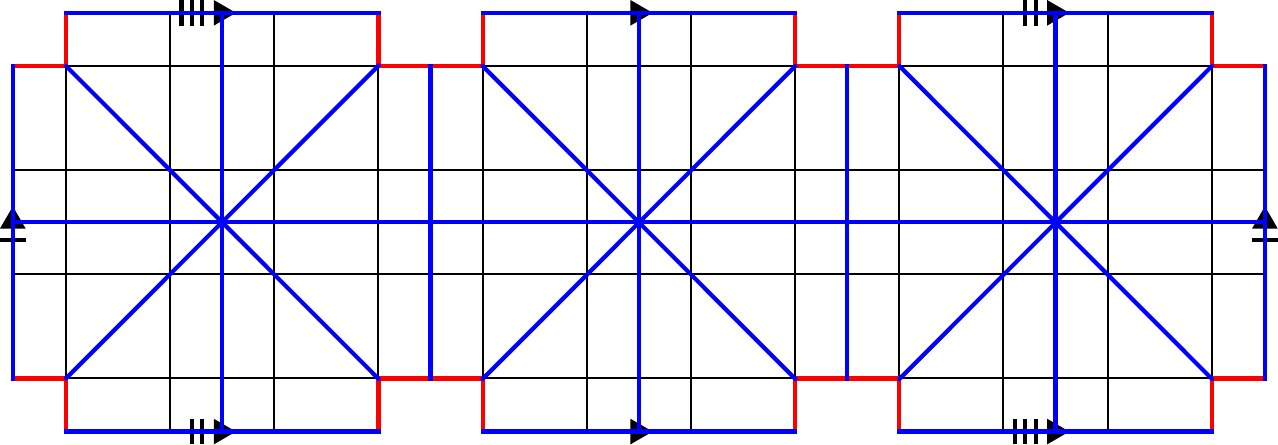}
    \caption{
    The subdivision $D_{2,1}(3)$ of $K_{2,1}(3)$ with the added hyperplanes colored blue.
    Note that each of the resulting cells are polytopes, and in each open $2$-cell $C$ the distance from $z\in C$ to $\partial K_{2,1}(3)$ is determined by a single horizontal (or vertical) measurement.
    Moreover, this distance is a linear function of the horizontal, resp. vertical, position of $z$ in $C$.
    }
    \label{fig:K21(3)firstrefinement}
\end{figure}

Let $D^{2}_{g,1}(n)$ be the product of the two copies of the complex $D_{g,1}(n)$.
By the previous paragraph, each particular metric function $d_{1}$ and $d_{2}$ is affine Morse--Bott on $D^{2}_{g,1}(n)$.
Refine this subdivision by taking the refinement with the subdivision $Q_{2}$ defined in Proposition \ref{distanct function between two points is affine Morse--Bott} to ensure that the particular metric function $d_{1,2}$ is also affine Morse--Bott.
Call this subdivision $Q'_{2}$.
Since the particular metric functions $d_{1}$, $d_{2}$, and $d_{1,2}$ are affine Morse--Bott on each of the $4$-cells of $Q'_{2}$, moving $z_{1}$ or $z_{2}$ results in affine change to $d_{1}$ and $d_{1,2}$ or $d_{1}$ and $d_{1,2}$, respectively.
It follows that the region of any $4$-cell in which two of these are equal is defined via a hyperplane, so by further subdividing by these hyperplanes we see that we get a refinement $Q$ on which $\theta_{2}$ is uniquely determined by exactly one of $d_{1}$, $d_{2}$, and $d_{12}$ on each open $4$-cell, and at the intersection of the closures of any two of these $4$-cells the corresponding particular metric functions agree.
Since the particular metric functions are affine Morse--Bott, $\theta_{2}$ is defined by exactly one particular metric function on each open $4$-cell of $Q$, and on adjacent $4$-cells the corresponding particular metric functions agree on their intersection of their boundaries, we see that $\theta_{2}$ is affine Morse--Bott on $Q$.
\end{proof}

Finally, we use the two previous propositions to prove that $\theta_{m}$ is affine Morse--Bott for all $m$, a fact that we will need to prove Lemma \ref{retract from unit square to small square}, which shows that we can retract $SF_{m}\big(K(n), r_{1}\big)$ onto $SF_{m}\big(K(n), r_{1}\big)$ for small $r_{1}<r_{2}$.

\begin{prop}\label{tautological function is affine Morse--Bott}
If $K(n)$ is a complex of the form $K_{g,0}(n)$ or $K_{g,1}(n)$, the tautological function $\theta_{m}:\big(K(n)\big)^{m}\to \R$ is affine Morse--Bott with respect to some subdivision $S$ of the complex $\big(K(n)\big)^{m}$.
\end{prop}

\begin{proof}
If $K(n)=K_{g,0}(n)$, let $Q$ be the decomposition of $\big(K(n)\big)^{2}$ defined in Proposition \ref{distanct function between two points is affine Morse--Bott}, whereas if $K(n)=K_{g,1}(n)$,  let $Q$ be the decomposition of $\big(K(n)\big)^{2}$ defined in Proposition \ref{tautological function is affine Morse--Bott on 2 points}.
For each pair $k,l$ such that $1\le k< l\le n $ take a copy of the subdivision $Q$ of the product $K(n)_{k}\times K(n)_{l}$, and denote it by $Q_{k,l}$.
We extend $Q_{k,l}$ to a subdivision of $\big(K(n)\big)^{m}$ in the obvious way.
Taking the common refinement of all these subdivisions yields a subdivision $S'$ of $\big(K(n)\big)^{m}$.
All of the particular metric functions are affine Morse--Bott on the cells of this subdivision.
Since these functions are affine we can further subdivide $S'$ via hyperplanes to get a polytopal structure such that on each open $2m$-cell of $S$ the tautological function is defined by a single particular metric function; moreover, the defining particular metric functions agree on the mutual boundaries of the $2m$-cells.
Therefore, $S$ is a polytopal decomposition of $K(n)$ such that $\theta_{m}$ is an affine Morse--Bott function.
\end{proof}

Now that we know that $\theta_{m}(\textbf{z})$ is affine Morse--Bott on $\big(K(n)\big)^{m}$, we describe notions of regular and critical values.
We will prove that the smallest nonzero critical value of $\theta_{m}(\textbf{z})$ is at least $\frac{1}{2}$; this will yield the homotopy equivalence of $F_{m}(\Sigma_{g,b})$ and $SF_{m}\big(K_{g,b}(n),d\big)$ for all $d\le 1$.
We come to these definitions by considering the tangent space of a neighborhood of a configuration $\textbf{z}$ in $F_{m}\big(K(n)\big)\subset \big(K(n)\big)^{m}$ and determining whether we can increase $\theta_{m}$ for all points in this neighborhood by flowing along some smooth vector field.
Namely, for each configuration $\textbf{z}$ and each particular metric function $f$, we will define a subset of the tangent space of $F_{m}\big(K(n)\big)$ at $\textbf{z}$, called the \emph{open cone} of $f$ at $\textbf{z}$, which we denote by $\text{Cone}_{\textbf{z}}(f)$.
We then take the intersection of these cones to get the open cone of $\theta_{m}$, which we will relate to its critical values.

First, we must make a few remarks about the notion of the tangent space of $\big(K(n)\big)^{m}$.
If $K(n)$ is of the form $K_{1,0}(n)$ or $K_{g,1}(n)$, every point $z$ in the interior of $K(n)$ has a well-defined tangent space isomorphic to $\R^{2}$, and it follows that $T_{\textbf{z}}F_{m}\big(K(n)\big)$ is well-defined and isomorphic to $\R^{2m}$ for every configuration not containing points in $\partial K(n)$.
Note that for $g>1$, the tangent space $TK_{g,0}(n)$ is not well-defined due to the cone point $p$.
It follows that the tangent space of $F_{m}\big(K_{g,0}(n)\big)$ is not well-defined either.
Fortunately, there is a simple workaround: If $z\neq p$, set $T_{z}K_{g,0}(n)\simeq \R^{2}$ as usual, and at $p$ set $T_{p}K_{g,0}(n)$ to be the trivial subspace $\{0\}\subset \R^{2}$.
Doing so lets $T_{\textbf{z}}F_{m}\big(K(n)\big)$ be the usual $\R^{2m}$ if the configuration $\textbf{z}$ does not contain $p$, and sets $T_{\textbf{z}}F_{m}\big(K(n)\big)$ to be $\R^{2m-2}\subset \R^{2m}$ if $p\in \textbf{z}$, by demanding that the $2$-dimensional summand corresponding to the point at $p$ be $\{0\}$.
While this definition warrants some concern, it will not be an issue due to the nature of the vector field we construct.
Namely, we construct a vector field on $F_{m}\big(K_{g,0}(n)\big)$ such that as a point in a configuration approaches $p$, the corresponding components of the vector field tend to $0$.
Alternatively, one could stratify $F_{m}\big(K_{g,0}(n)\big)$ into configurations that contain $p$ and those that do not.
The former has tangent space isomorphic to $\R^{2m-2}$, while the latter has tangent space isomorphic to $\R^{2m}$; moreover, the flow arising from our smooth vector field on these tangent spaces preserves this stratification.
With this in mind, we turn to determining which small movements of points in a configuration lead to increases in $\theta_{m}$, as successively performing such movements would yield a retraction of $F_{m}\big(K(n)\big)$ onto $SF_{m}\big(K(n)\big)$.

To figure out which movements of the points in $\textbf{z}$ lead to an increase in $\theta_{m}$, we consider the subspaces of $T_{\textbf{z}}F_{m}\big(K(n)\big)$ that correspond to movements of points increasing the particular metric functions.
For a particular metric function $f$, we write $\text{Cone}_{\textbf{z}}(f)$ for the subspace of $T_{\textbf{z}}F_{m}\big(K(n)\big)$ corresponding to movements of points in $\textbf{z}$ increasing $f$.

If $f$ is of the form $d_{i,j}$ and $d_{i,j}(\textbf{z})=d_{\infty}(z_{i},z_{j})$ is realized by only one of the horizontal and vertical distances between $z_{i}$ and $z_{j}$, this distance is not maximal, and neither $z_{i}$ nor $z_{j}$ is $p$, then $\text{Cone}_{\textbf{z}}(f)$ is isomorphic (via an affine map) to $\R^{2m-4}\times G$ where $G$ is the open half-space in $\R^{4}\cong \R^{2}\times \R^{2}$ corresponding to movements of the points $z_{i}$ and $z_{j}$ that increase the Cheybshev distance between the two points.
If one of $z_{i}$ and $z_{j}$ is $p$, then $\text{Cone}_{\textbf{z}}(f)$ is isomorphic (via an affine map) to $\R^{2m-4}\times F$ where $F$ is the open half-space in $\R^{2}\subset \R^{4}\cong 0\times \R^{2}\subset \R^{2}\times \R^{2}$ corresponding to movements of the point not at $p$ that increase the Cheybshev distance between the two points.
If $d_{\infty}(z_{i}, z_{j})$ is simultaneously realized by both the vertical and horizontal distances between $z_{i}$ and $z_{j}$ and neither is $p$, then $\text{Conf}_{\textbf{z}}(f)$ is isomorphic (via an affine map) to $\R^{2m-4}\times (G_{1}\times \R^{2})$, where $G_{1}$ is the open sector in $\R^{2}$ with angle $\frac{3\pi}{2}$ corresponding to movements of $z_{i}$ and $z_{j}$ that increase the Chebyshev distance between them.
If one of $z_{i}$ and $z_{j}$ is $p$, then $\text{Conf}_{\textbf{z}}(f)$ is isomorphic (via an affine map) to $\R^{2m-4}\times F_{1}\times (0\in \R^{2})$, where $F_{1}$ is the open sector in $\R^{2}$ of angle $\frac{3\pi}{2}$ corresponding to movements of the non-$p$-point that that increase the Chebyshev distance between the pair.

Similarly, if $f$ of the form $d_{i}$, the Chebyshev distance between $z_{i}$ and $\partial K_{g,1}(n)$ is realized by an interval that is not a single point in the boundary, and $d\big(z_{i},\partial K_{g,1}(n)\big)$ is not maximal, then $\text{Cone}_{\textbf{z}}(f)$ is isomorphic (via an affine map) to $\R^{2m-2}\times H$, where $H$ is an open half-plane in $\R^{2}$ corresponding to movements of the point $z_{i}$ that increase the Chebyshev distance between it and the boundary. 
If $d_{i}(\textbf{z})=d_{\infty}\big(z_{i}, \partial K_{g,1}(n)\big)$ is realized by a single point---necessarily a corner of $\partial K_{g,1}(n)$---then $\text{Cone}_{\textbf{z}}(f)$ is isomorphic (via an affine map) to $\R^{2m-2}\times G_{1}$, where $G_{1}$ is the open sector in $R^{2}$ of angle $\frac{3\pi}{2}$ corresponding to movements of $z_{i}$ that increase $d_{\infty}\big(z_{i}, \partial K_{g,1}(n)\big)$.

If $f$ is either type of particular metric function and there is no way to increase $f$ by a local move, we set $\text{Cone}_{\textbf{z}}(f)=\emptyset$.

Given a point $\textbf{z}\in F_{m}\big(K(n)\big)$, we define the \emph{open cone} of $\theta_{m}$ at $\textbf{z}$ by setting $\text{Cone}_{\textbf{z}}(\theta_{m})$ to be the intersection section of the open cones of the particular metric functions realizing the $\theta_{m}(\textbf{z})$ in $\R^{2m}$, i.e.,
\[
\text{Cone}_{\textbf{z}}(\theta_{m}):=\text{Cone}_{\textbf{z}}(f_{1})\cap\cdots\cap\text{Cone}_{\textbf{z}}(f_{l}).
\]
It follows that $\text{Cone}_{\textbf{z}}(\theta_{m})$ consists of the vectors in $T_{\textbf{z}}F_{m}\big(K(n)\big)$ that correspond to the movements of the points in $\textbf{z}$ that increase $\theta_{m}$.

If $\textbf{z}\in \big(K(n)\big)^{m}$ is such that $\theta_{m}(\textbf{z})>0$, then we say that it is a \emph{regular point} of $\theta_{m}$ if there is a neighborhood $U_{\textbf{z}}$ of $\textbf{z}$ in $F_{m}\big(K(n)\big)$ such that for every point $\textbf{z}'\in U_{\textbf{z}}$, the open cone of $\theta_{m}$ at $\textbf{z}'$ is nonempty.
Otherwise, we say that $\textbf{z}$ a \emph{critical point} of $\theta_{m}$.
For convenience, all configurations $\textbf{z}$ in $\Delta\cup \partial\Big(\big(K(n)\big)^{m}\Big)\subset \big(K(n)\big)^{m}$ are defined to be critical, where $\Delta$ is the fat diagonal.

\begin{figure}[h]
    \centering
    \includegraphics[width=.75\linewidth]{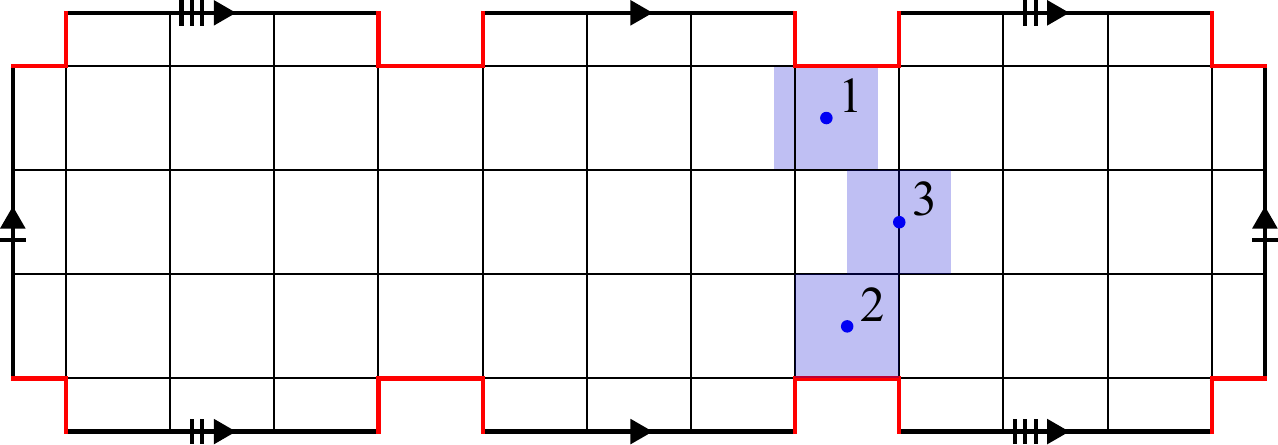}
    \caption{
    A critical point of the function $\theta_{3}:\big(K_{2,1}(3)\big)^{3}\to \R$ with squares of side length twice the critical value.
    Note that the squares are tangent to each other and the sides of $K_{2,1}(3)$.
    There is no way to slightly move the points and get a larger value of $\theta_{3}$, as vertical movement would make smash the squares into each other or the boundary of $K_{2,1}(3)$, and small horizontal movements keep the squares tangent to each other and the boundary.
    }
    \label{fig:K21(3)criticalpoint}
\end{figure}

One can think of regular points of $\theta_{m}$ as the configurations $\textbf{z}$ for which there exists a well-defined way to spread out the $m$ points constituting $\textbf{z}$ that that can be extended to nearby configurations.
We will see that if $K(n)$ is one of the complexes $K_{g,0}(n)$ or $K_{g,1}(n)$ and $m\le n$, then all $r\in \big(0,\frac{1}{2}\big)$ are regular values of $\theta_{m}$.
We will use this to prove $F_{m}(\Sigma_{g,b})$ is homotopy equivalent to $SF_{m}\big(K_{g,b}(n)\big)$.

We begin by noting that if two points are not close to each other in a configuration, then they are not close to each other in nearby configurations.
To formalize this idea we recall the definition of the contact graph.
If $K(n)$ is of the form $K_{g,0}(n)$ or $K_{g,1}(n)$ and $\textbf{z}$ is a configuration in $F_{m}\big(K(n)\big)$ such that $\theta_{m}(\textbf{z})=r$, the \emph{contact graph of $\textbf{z}$}, which we denote by $\Gamma_{\textbf{z}}$, is the embedded graph in $K$ such that there is an \emph{internal vertex} at each coordinate $z_{i}$ of $\textbf{z}$ and an \emph{internal edge} of Chebyshev length $2r$ between each pair of vertices $z_{i}$ and $z_{j}$ such that $d_{\infty}(z_{i}, z_{j})=2r$.
Additionally, if the $l^{\infty}$-ball of radius $r$ centered at $z_{i}$ is tangent to $\partial K$ along a closed interval, then the center of this interval is an \emph{external vertex} and there is an \emph{external edge} of Chebyshev length $r$ between this external vertex and the internal vertex at $z_{i}$.
For an example, see Figure \ref{fig:contactgraphK11(7)}.

\begin{figure}[h]
    \centering
    \includegraphics[width=.5
    \linewidth]{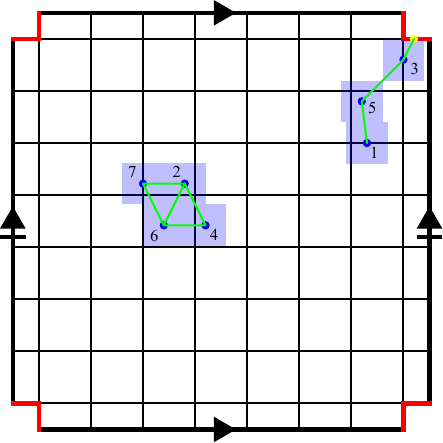}
    \caption{
The contact graph of a configuration $\textbf{z}$ in $F_{7}\big(K_{1,1}(7)\big)$, with $\theta_{7}(\textbf{z})=\frac{2}{5}$.
We have colored the internal vertices, which correspond to the points in the configuration, blue, the sole external vertex yellow, and the edges green.
We have used blue to draw the $l^{\infty}$-balls of radius $\frac{2}{5}$ around the points constituting $\textbf{z}$.}
    \label{fig:contactgraphK11(7)}
\end{figure}

Given a contact graph $\Gamma_{\textbf{z}}$ in $K(n)$ one can interpret its edges as vectors.
Namely, if $z_{i}=(x_{i}, y_{i})$ and $z_{j}=(x_{j}, y_{j})$ are connected by an internal edge, then by Lemmas \ref{section of small set no boundary} and \ref{section of small set with boundary} we can lift this edge to $\tilde{K}(n)$ and then project the image of this lift to $\R^{2}$.
The lift is an isometry, though the projection might not be if the edge passes through $p$, the point of non-zero curvature in $K_{g,0}(n)$---this does not depend on choice of lift.
Let $\hat{z}_{i}=(\hat{x}_{i}, \hat{y}_{i})$ and $\hat{z}_{j}=(\hat{x}_{j}, \hat{y}_{j})$ be the images of $z_{i}$ and $z_{j}$ in $\R^{2}$.
If the projection is a horizontal, resp. vertical, isometry, then the horizontal, resp. vertical, component of the edge is represented by the $2$-vector $(\hat{x}_{i}-\hat{x}_{j}, \hat{x}_{j}-\hat{x}_{i})$, resp. $(\hat{y}_{i}-\hat{y}_{j}, \hat{y}_{j}-\hat{y}_{i})$.
If the projection is not a horizontal, resp. vertical, isometry, then the horizontal, resp. vertical, component of the edge is represented by the $2$-vector $(\hat{x}_{i}, \hat{x}_{j})$ and $(\hat{y}_{i}, \hat{y}_{j})$.
Similarly, if $z_{i}=(x_{i}, y_{i})$, an internal vertex, is connected to an external vertex $w=(u, v)$, then we consider the $1$-vectors $(\hat{x}_{i}-\hat{u})$ and $(\hat{y}_{i}-\hat{v})$ arising from the projecting the image of the lift of this external edge to $\R^{2}$.

This vector interpretation of the edges of $\Gamma_{\textbf{z}}$ lends itself well to determining whether a vector $\vec{\textbf{v}}_{\textbf{z}}\in T_{\textbf{z}}F_{m}\big(K(n)\big)$ lies in $\text{Cone}_{\textbf{z}}(\theta_{m})$.
Namely, if there be an internal edge in $\Gamma_{\textbf{z}}$ between the internal vertices $z_{i}$ and $z_{j}$, i.e., $\frac{1}{2}d_{\infty}(z_{i}, z_{j})=\theta_{m}(\textbf{z})$ and $d_{\infty}(z_{i}, z_{j})$ is realized by the horizontal, resp. vertical, distance between $z_{i}$ and $z_{j}$, and the projection is a horizontal, resp. vertical, isometry, then $\vec{\textbf{v}}_{\textbf{z}}\in \text{Cone}_{\textbf{z}}(d_{i,j})$ if and only if $\big\langle (\vec{\textbf{v}}_{\textbf{x}, \textbf{i}},\vec{\textbf{v}}_{\textbf{x}, \textbf{j}}), (\hat{x}_{i}-\hat{x}_{j}, \hat{x}_{j}-\hat{x}_{i}) \big\rangle>0$, resp. $\big\langle (\vec{\textbf{v}}_{\textbf{y}, \textbf{i}},\vec{\textbf{v}}_{\textbf{y}, \textbf{j}}), (\hat{y}_{i}-\hat{y}_{j}, \hat{y}_{j}-\hat{y}_{i}) \big\rangle>0$.
Here $\vec{\textbf{v}}_{\textbf{x}}$, resp. $\vec{\textbf{v}}_{\textbf{y}}$, denotes the horizontal, resp. vertical, component of $\vec{\textbf{v}}_{\textbf{z}}$.
If the projection is not an horizontal, resp. vertical, isometry,  then $\vec{\textbf{v}}_{\textbf{z}}\in \text{Cone}_{\textbf{z}}(d_{i,j})$ if and only if $\big\langle (\vec{\textbf{v}}_{\textbf{x}, \textbf{i}},\vec{\textbf{v}}_{\textbf{x}, \textbf{j}}), (\hat{x}_{i}, \hat{x}_{j}) \big\rangle>2\theta_{m}(\textbf{z})$, resp. $\big\langle (\vec{\textbf{v}}_{\textbf{y}, \textbf{i}},\vec{\textbf{v}}_{\textbf{y}, \textbf{j}}), (\hat{y}_{i}, \hat{y}_{j}) \big\rangle>2\theta_{m}(\textbf{z})$.

Similarly, let $z_{i}$ be an internal vertex of $\Gamma_{\textbf{z}}$ connected by an external edge to the external vertex $w=(u,v)\in\partial K_{g,1}(n) $, that is, $d_{\infty}\big(z_{i},\partial K_{g,1}(n)\big)=d_{\infty}(z_{i},w)=\theta_{m}(\textbf{z})$.
If $d_{\infty}(z_{i},w)$ is realized by the horizontal, resp. vertical, distance between $z_{i}$ and $w$, then $\vec{\textbf{v}}_{\textbf{z}}\in \text{Cone}_{\textbf{z}}(d_{i})$ if and only if 
$\big\langle (\vec{\textbf{v}}_{\textbf{x}, \textbf{i}}), (\hat{x}_{i}-\hat{u}) \big\rangle>0$, resp. $\big\langle (\vec{\textbf{v}}_{\textbf{y}, \textbf{i}}), (\hat{y}_{i}-\hat{v}) \big\rangle>0$.

In general, these relations are equivalent to the statement that the horizontal, resp. vertical, distance between adjacent vertices of the contact graph increases if the corresponding points $z_{i}$ and $z_{j}$ in the configuration move a little in the directions $\vec{\textbf{v}}_{\textbf{x}, \textbf{i}}$ and $\vec{\textbf{v}}_{\textbf{x}, \textbf{j}}$, resp. $\vec{\textbf{v}}_{\textbf{y}, \textbf{i}}$ and $\vec{\textbf{v}}_{\textbf{y}, \textbf{j}}$.
It follows that $\vec{\textbf{v}}_{\textbf{z}}$ is in $\text{Cone}_{\textbf{z}}(\theta_{m})$ if and only if these dot product relations hold for every vector coming from the contact graph $\Gamma_{\textbf{z}}$.

To construct our flow that leads to a deformation retract of $F_{m}\big(K(n)\big)$ onto $SF_{m}\big(K(n)\big)$, we find a function increasing $\theta_{m}$ on a neighborhood of every $\textbf{z}$ such that $\theta_{m}(\textbf{z})<\frac{1}{2}$.
It would be convenient if nearby configurations to behaved similarly.
This is the case, though in order to formalize this we need a notion of nearby configurations.
Fortunately, the Chebyshev metric on $K(n)$ induces a metric $d$ on $\big(K(n)\big)^{m}$, hence $F_{m}\big(K(n)\big)$, i.e., 
\[
d(\textbf{z}, \textbf{z}'):=\max\big\{d_{\infty}(z_{i}, z'_{i})\big\},
\]
and we use this metric to define open balls in $F_{m}\big(K(n)\big)$.

In the following proposition we show that if one takes a small enough ball $B_{R}(\textbf{z})$ around a configuration  $\textbf{z}\in F_{m}(K)$, then, after forgetting the metric, there is a natural inclusion of $\Gamma_{\textbf{z}'}$ into $\Gamma_{\textbf{z}}$ for all $\textbf{z}'\in B_{R}(\textbf{z})$.
This proves that if two points are not close in a configuration, then they are not close in nearby configurations.

\begin{prop}\label{nearby points have similar contact graphs}
For $K(n)$ of the form $K_{g,0}(n)$ or $K_{g,1}(n)$, let $\textbf{z}$ be a configuration in $F_{m}\big(K(n)\big)$ such that $\theta_{m}(\textbf{z})=r<\frac{1}{2}$, and set 
\[
r':=\frac{1}{2}\min_{i, j}\big\{d_{\infty}(z_{i}, z_{j})|d_{\infty}(z_{i}, z_{j})\neq 2r\big\} \indent\text{and}\indent r'':=\min_{i}\Big\{d_{\infty}\big(z_{i}, \partial K(n)\big)|d_{\infty}\big(z_{i}, \partial K(n)\big)\neq r\Big\}.
\]
If $R=\frac{1}{16}\min\{r, 1-r, r'-r, r''-r\}$, then, for all $\textbf{z}'\in B_{R}(\textbf{z})$, the contact graph $\Gamma_{\textbf{z}'}$ embeds as a subgraph of $\Gamma_{\textbf{z}}$ upon forgetting the metric via the map that sends $z'_{i}$ to $z_{i}$.
\end{prop}

\begin{proof}
We begin by determining an upper bound for the tautological function at $\textbf{z}'$.
If $\frac{1}{2}d_{\infty}(z_{i}, z_{j})=\theta_{m}(\textbf{z})=r$, then
\begin{align*}
d_{\infty}(z'_{i}, z'_{j})&\le d_{\infty}(z'_{i}, z_{i})+d_{\infty}(z_{i}, z_{j})+d_{\infty}(z_{j}, z'_{j})\\
&\le R+2r+R.
\end{align*}
Similarly, if $d_{\infty}\big(z_{i}, \partial K(n)\big)=\theta_{m}(\textbf{z})=r$, then
\begin{align*}
d_{\infty}\big(z'_{i}, \partial K(n)\big)&\le d_{\infty}(z'_{i}, z_{i})+d_{\infty}\big(z_{i}, \partial K(n)\big)\\
&\le R+r.
\end{align*}
It follows that $\theta_{m}(\textbf{z}')\le r+R$.

It remains to check that if there is no internal edge between $z_{i}$ and $z_{j}$, then there is no internal edge between $z'_{i}$ and $z'_{j}$, and if there is no external edge at $z_{i}$, then there is no external edge at $z'_{i}$.
This is equivalent to showing that if $d_{\infty}(z_{i}, z_{j})$ does not realize $\theta_{m}(\textbf{z})$, then $d_{\infty}(z'_{i}, z'_{j})>2r+2R$, and if $d_{\infty}(z_{i}, \partial K)$ does not realize $\theta_{m}(\textbf{z})$, then $d_{\infty}(z'_{i}, \partial K)>r+R$.
If $d_{\infty}(z_{i}, z_{j})$ does not realize $\theta_{m}(\textbf{z})$, then $d_{\infty}(z_{i}, z_{j})\ge2r'$ by assumption.
Therefore, 
\begin{align*}
d_{\infty}(z'_{i}, z'_{j})&\ge -d_{\infty}(z'_{i}, z_{i})+d_{\infty}(z_{i}, z_{j})-d_{\infty}(z_{j}, z'_{j})\\
&\ge -R+2r'-R\\
&>2r+2R.
\end{align*}
Similarly, if $d_{\infty}\big(z_{i}, \partial K(n)\big)$ does not realize $\theta_{m}(\textbf{z})$, then $d_{\infty}\big(z_{i}, \partial K(n)\big)\ge r''$, so 
\begin{align*}
d_{\infty}\big(z'_{i}, \partial K(n)\big)&\ge -d_{\infty}(z'_{i}, z_{i})+d_{\infty}\big(z_{i}, \partial K(n)\big)\\
&\ge -R+r''\\
&>r+R.
\end{align*}
Thus, $\theta_{m}(\textbf{z}')$ can only be realized by the $\frac{1}{2}d_{\infty}(z'_{i}, z'_{j})$ such that $\frac{1}{2}d_{\infty}(z_{i}, z_{j})$ realizes $\theta_{m}(\textbf{z})$ and by the $d_{\infty}(z'_{i}, \partial K)$ such that $d_{\infty}(z_{i}, \partial K)$ realizes $\theta_{m}(\textbf{z})$.
Since the edges of $\Gamma_{\textbf{z}'}$ correspond to the $\frac{1}{2}d_{\infty}(z'_{i}, z'_{j})$ and the $d_{\infty}(z'_{i}, \partial K)$ realizing $\theta_{m}(\textbf{z}')$, we see that for all $\textbf{z}'\in B_{R}(\textbf{z})$ the map $z'_{i}\mapsto z_{i}$ induces an inclusion of graphs $\Gamma_{\textbf{z}'}\hookrightarrow\Gamma_{\textbf{z}}$ upon forgetting the metric.
\end{proof}

We use Proposition \ref{nearby points have similar contact graphs} to find a neighborhood of $\textbf{z}$ on which there exists a $\theta_{m}$ increasing function.
We begin by considering the case in which the contact graph of $\textbf{z}$ is connected.

\begin{prop}\label{move points away from boundary}
For $m\le n$ and $g\ge 1$, let $\textbf{z}=(z_{1}, \dots, z_{m})$ be a configuration in $K(n)$, where $K(n)$ is a cube complex of the form $K_{g,0}(n)$ or $K_{g,1}(n)$, such that $\theta_{m}(\textbf{z})=r<\frac{1}{2}$ and such that the contact graph $\Gamma_{\textbf{z}}$ is connected.
Additionally, set 
\[
r':=\frac{1}{2}\min_{i, j}\big\{d_{\infty}(z_{i}, z_{j})|d_{\infty}(z_{i}, z_{j})\neq 2r\big\} \indent\text{and}\indent r'':=\min_{i}\Big\{d_{\infty}\big(z_{i}, \partial K(n)\big)|d_{\infty}\big(z_{i}, \partial K(n)\big)\neq r\Big\}.
\]
If $R\le \frac{1}{16}\min\{r, 1-r, r'-r, r''-r\}$, then there is a smooth map $\phi_{R}:B_{R}(\textbf{z})\to F_{m}(K)$ such that $\theta_{m}\big(\phi_{\textbf{z}}(\textbf{z}')\big)>\theta_{m}(\textbf{z}')$ for all $\textbf{z}'\in B_{R}(\textbf{z})$.
\end{prop}

\begin{proof}
Note that if $m=1$ and $K(n)=K_{g,0}(n)$, then $\theta_{m}$ is not defined.
In this case we set $\phi_{R}$ to be the identity map.
Otherwise, we proceed as follows.

Since $m\le n$ and $r<\frac{1}{2}$, it follows that the longest non-backtracking path in $\Gamma_{\textbf{z}}$ has length at most $2mr<n$.
It follows from Proposition \ref{injectivity radius} that if $K(n)=K_{g,0}(n)$, the contact graph $\Gamma_{\textbf{z}}$ is contractible in $K(n)$; if $K(n)=K_{g,1}(n)$, then $\Gamma_{\textbf{z}}$ is homotopic into the boundary.
Moreover, by our constraints on $R$, the $2R$-neighborhood of $\Gamma_{\textbf{z}}$ is either contractible or homotopic into the boundary, and it has diameter at most $2mr+4R<n$.
It follows that the $2R$-neighborhood of $\Gamma_{\textbf{z}}$ satisfies the conditions of Lemma \ref{section of small set no boundary} if $K(n)=K_{g,0}(n)$ and Lemma \ref{section of small set with boundary} if $K(n)=K_{g,1}(n)$, so there exists a section of $\pi:\tilde{K}(n)\to K(n)$ on this neighborhood.
Moreover, if $K(n)=K_{g,1}(n)$, we can choose our section to be such that its image is Chebyshev distance at least $r''$ from the outer boundary of $\partial\tilde{K}(n)$, that is the boundary component further from the origin. 
Each of these sections restricts to a section of $\pi$ on $\bigcup_{i=1}^{m} B_{R}(z_{i})$, and the image of such a section is $\bigcup_{i=1}^{m} B_{R}(\tilde{z}_{i})$, where $\tilde{z}_{i}$ is the image of $z_{i}$.

By the definition of the square $S(n)$, there is a projection map from $\tilde{K}(n)$ to $\R^{2}$, such that the image of $\bigcup_{i=1}^{m} B_{R}(\tilde{z}_{i})$ is $\bigcup_{i=1}^{m} B_{R}(\hat{z}_{i})$, where $\hat{z}_{i}$ is the projection of $\tilde{z}_{i}$ to $\R^{2}$.
Setting $L_{R}=\sup_{\hat{z}\in \bigcup_{i=1}^{m} B_{R}(\hat{z}_{i})} \Big\{d_{2}\big(\hat{z}, (0,0)\big)\Big\}$ and $t_{R}=\frac{L_{R}+R}{L_{R}}$, we define the map $\hat{\phi}_{R}:\bigcup_{i=1}^{m} B_{R}(\hat{z}_{i})\to S(n)\subset \R^{2}$ via $\hat{\phi}_{R}(\hat{z})=t_{R}\hat{z}$. 
This map lifts to a map $\tilde{\phi}_{R}:\bigcup_{i=1}^{m} B_{R}(\tilde{z}_{i})\to \tilde{K}(n)$, which descends to a map $\phi_{R}:\bigcup_{i=1}^{m} B_{R}(z_{i})\to K(n)$, which in turn induces a map, which we also call $\phi_{R}$, from $B_{R}(\textbf{z})$ to $F_{m}\big(K(n)\big)$. 
Note that these maps move points in a configuration a distance at most $R$.
Since $\phi_{R}$ is affine it is smooth; moreover, we claim that it increases $\theta_{m}$ on $B_{R}(\textbf{z})$.

To see this, note that if $d_{\infty}(z'_{i}, z'_{j})$ realizes $\theta_{m}(\textbf{z}')$ for $\textbf{z}'\in B_{R}(\textbf{z})$, then $d_{\infty}\big(\phi_{R}(z'_{i}), \phi_{R}(z'_{j})\big)=t_{R}d_{\infty}(z'_{i}, z'_{j})$.
This follows from the fact that the map $\hat{\phi}_{R}$ scales distances in $S(n)$ by $t_{R}$, and the projection map from $\tilde{K}(n)$ to $K(n)$ is an isometry on the open $l^{\infty}$-ball of radius $2$ around each point.
Similar arguments show that if $d_{\infty}\big(z'_{i}, \partial K(n)\big)$ realizes $\theta_{m}(\textbf{z}')$, then $d_{\infty}\big(\phi_{R}(z'_{i}), \partial K(n)\big)\ge t_{R}d_{\infty}(z'_{i}, \partial K(n))$, with equality if and only if $m=1$.
Since $\phi_{R}$ moves points a Chebyshev distance at most $R$ and  $\theta_{m}(\textbf{z}')\le r+R$ for all $\textbf{z}'\in B_{R}(\textbf{z})$, Proposition \ref{nearby points have similar contact graphs} proves that $\theta_{m}\big(\phi_{R}(\textbf{z}')\big)\le r+2R$.
Therefore, it suffices to check that pairs of points that start far apart, stay far apart and that points that start far from the boundary, stay from the boundary.

The proof of Proposition \ref{nearby points have similar contact graphs} shows that if $\frac{1}{2}d_{\infty}(z'_{i}, z'_{j})$ does not realize $\theta_{m}(\textbf{z}')$, then $\frac{1}{2}d_{\infty}(z'_{i}, z'_{j})\ge r'-R$.
Since $\phi_{R}$ moves points a distance at most $R$, it follows that
\begin{align*}
d_{\infty}\big(\phi_{R}(z'_{i}), \phi_{R}(z'_{j})\big)&\ge -d_{\infty}\big(\phi_{R}(z'_{i}), z'_{i}\big)+d_{\infty}(z'_{i}, z'_{j})-d_{\infty}\big(z'_{j}, \phi_{R}(z'_{j})\big)\\
&\ge -R+2r'-2R-R\\
&>2r+4R.
\end{align*}
Similarly, the proof of Proposition \ref{nearby points have similar contact graphs} also shows that if $d_{\infty}\big(z'_{i}, \partial K_{g,1}(n)\big)$ does not realize $\theta_{m}(\textbf{z}')$, then $d_{\infty}\big(z'_{i}, \partial K_{g,1}(n)\big)\ge r''-R$.
Therefore,
\begin{align*}
d_{\infty}\big(\phi_{R}(z'_{i}), \partial K_{g,1}(n)\big)&\ge-d_{\infty}\big(\phi_{R}(z'_{i}), z'_{i}\big)+d_{\infty}\big(z'_{i}, \partial K_{g,1}(n)\big)\\
&\ge -R+r''-R\\
&>r+2R,
\end{align*}
proving that $\theta_{m}\big(\phi_{R}(\textbf{z})\big)$ can only be realized by the $\frac{1}{2}d_{\infty}\big(\phi_{R}(z'_{i}), \phi_{R}(z'_{j})\big)$ such that $\frac{1}{2}d_{\infty}(z'_{i}, z'_{j})$ realizes $\theta_{m}(\textbf{z}')$ and the $d_{\infty}\big(\phi_{R}(z'_{i}), \partial K_{g,1}(n)\big)$ such that $d_{\infty}\big(z'_{i}, \partial K_{g,1}(n)\big)$ realizes $\theta_{m}(\textbf{z}')$.
Moreover, as $\theta_{m}(\textbf{z}')\le r+R<\frac{1}{2}$, the map $\phi_{R}$ must increase these these particular metric functions by a factor of least $t_{R}$ for all $\textbf{z}'\in B_{R}(\textbf{z})$.
Therefore, $\phi_{R}$ is a smooth function on $B_{R}(\textbf{z})$ that increases $\theta_{m}$ for all $\textbf{z}'\in B_{R}(\textbf{z})$.
\end{proof}

The map $\phi_{R}$ defined in Proposition \ref{move points away from boundary} yields a vector in $\text{Cone}_{\textbf{z}'}(\theta_{m})$ for all $\textbf{z}'\in B_{R}(\textbf{z})$.
Since $\phi_{R}$ moves points in a configuration only a small distance, we see that at $\textbf{z}'\in B_{R}(\textbf{z})$ the map $\phi_{R}$ corresponds to the vector $\vec{\phi}_{R}(\textbf{z}'):=\hat{\phi}_{R}(\hat{\textbf{z}}')-\hat{\textbf{z}}'$ in $T_{\hat{\textbf{z}}'}\R^{2m}$, where we use the additive structure of $\R^{2}$ in which $S(n)$ lies.
Moreover, this can naturally be thought of as vector in $T_{\textbf{z}'}F_{m}\big(K(n)\big)$.
As the map $\phi_{R}$ is smooth, we get a smooth vector field on $B_{\textbf{z}}(R)$.

If $z'_{i}$ and $z'_{j}$ are internal vertices in $\Gamma_{\textbf{z}'}$ connected by an internal edge, i.e., $\frac{1}{2}d_{\infty}(z'_{i}, z'_{j})=\theta_{m}(\textbf{z}')$ such that $d_{\infty}(z'_{i}, z'_{j})$ is realized by the horizontal, resp. vertical, distance between $z'_{i}$ and $z'_{j}$ and the projection map $\tilde{K}(n)\to \R^{2}$ is a horizontal, resp. vertical, isometry of the edge between $z'_{i}$ and $z'_{j}$, then we have that the dot product  $\big\langle (\vec{\phi}_{R,\textbf{x}}(\textbf{z}')_{i}, \vec{\phi}_{R,\textbf{x}}(\textbf{z}')_{j}), (\hat{x}'_{i}-\hat{x}'_{j}, \hat{x}'_{j}-\hat{x}'_{i}) \big\rangle$ is greater than $0$, resp. $\big\langle (\vec{\phi}_{R,\textbf{y}}(\textbf{z}')_{i},\vec{\phi}_{R,\textbf{y}}(\textbf{z}')_{i}),(\hat{y}'_{i}-\hat{y}'_{j}, \hat{y}'_{j}-\hat{y}'_{i}) \big\rangle>0$.
If the projection is not a horizontal, resp. vertical, isometry, then $\big\langle (\vec{\phi}_{R,\textbf{x}}(\textbf{z}')_{i}, \vec{\phi}_{R,\textbf{x}}(\textbf{z}')_{j}), (\hat{x}'_{i}, \hat{x}'_{j}) \big\rangle>2\theta_{m}(\textbf{z}')$, resp. $\big\langle (\vec{\phi}_{R,\textbf{y}}(\textbf{z}')_{i},\vec{\phi}_{R,\textbf{y}}(\textbf{z}')_{i}),(\hat{y}'_{i}, \hat{y}'_{j}) \big\rangle>2\theta_{m}(\textbf{z}')$.
Similarly, if $z'_{i}$ is an internal vertex of $\Gamma_{\textbf{z}'}$ connected by an external edge to the external vertex $w=(u,v)\in\partial K_{g,1}(n) $, that is, $d_{\infty}\big(z'_{i},\partial K_{g,1}(n)\big)=d_{\infty}(z'_{i},w)=\theta_{m}(\textbf{z}')$, such that this distance is realized by the horizontal, resp. vertical, distance between $z'_{i}$ and $w$, then
$\big\langle (\vec{\phi}_{R,\textbf{x}}(\textbf{z}')_{i}, (\hat{x}_{i}-\hat{u}) \big\rangle>0$, resp. $\big\langle (\vec{\phi}_{R,\textbf{y}}(\textbf{z}')_{i}, (\hat{y}_{i}-\hat{v}) \big\rangle>0$.

\begin{remark}\label{no need to worry about tangent space}
There is no need to be concerned about this vector field being well-defined in the case $K(n)=K_{g,0}(n)$, even though $\cup_{i=1}^{m}B_{R}(z_{i})$ might contain $p$, the point of non-zero curvature, making the notion of $T B_{R}(\textbf{z})$ a little fuzzy as a result---recall that if $p$ is one of the points in the configuration $\textbf{z}'\in B_{R}(\textbf{z})$, then we set $T_{\textbf{z}'} B_{R}(\textbf{z})=\R^{2m-2}\subset \R^{2m}$ by setting the coordinates corresponding to the point at $p$ of any vector equal to $0$.
To see this, note that $\phi_{R}$ moves points starting close to $p$ further away from $p$ by a factor of $t_{R}$, i.e., if $z'\in \cup_{i=1}^{m}B_{R}(z_{i})$ is such that $d_{\infty}(z', p)\le r+R$, then $d_{\infty}\big(\phi_{R}(z'), p\big)=t_{R}d_{\infty}(z', p)\ge d_{\infty}(z', p)$, and keeps $p$ still, i.e., if $p\in\cup_{i=1}^{m}B_{R}(z_{i})$, then $\phi_{R}(p)=p$.
It follows that as $z'_{i}$ goes to $p$, then corresponding components of $\vec{\phi}_{R}(\textbf{z}')$ go to $0$. 
Thus, despite the fact that $T B_{R}(\textbf{z})$ might not be well-defined the classical sense, our workaround provides a reasonable notion in which this vector field is smooth.
\end{remark}

Next, we show that the map $\phi_{R}$ defined in Proposition \ref{move points away from boundary} can be used to define a map that increases $\theta_{m}$ on a small neighborhood of any $\textbf{z}\in F_{m}\big(K(n)\big)$ such that $\theta_{m}(\textbf{z})<\frac{1}{2}$.

\begin{lem}\label{increase tautological function}
For $m\le n$ and $g\ge 1$, let $K(n)$ be a cube complex of the form $K_{g,0}(n)$ or $K_{g,1}(n)$, and let $\textbf{z}\in F_{m}\big(K(n)\big)$ be a configuration such that $\theta_{m}(\textbf{z})=r<\frac{1}{2}$.
There is a neighborhood $U_{\textbf{z}}$ of $\textbf{z}$ on which there exists a smooth function $\phi_{\textbf{z}}:U_{\textbf{z}}\to F_{m}\big(K(n)\big)$ such that $\theta_{m}\big(\phi_{\textbf{z}}(\textbf{z}')\big)>\theta_{m}(\textbf{z}')$ for all $\textbf{z}'\in U_{\textbf{z}}$.
\end{lem}

\begin{proof}
Given $\textbf{z}=(z_{1}, \dots, z_{m})\in F_{m}\big(K(n)\big)$, we partition $\{z_{1}, \dots, z_{m}\}$ into sets such that each set $\{z_{i}\}_{i\in I}$ corresponds to a distinct connected component of the contact graph $\Gamma_{\textbf{z}}$.
Ordering the set $\{z_{i}\}_{i\in I}$ yields a point in $\textbf{z}_{I}\in F_{|I|}\big(K(n)\big)$.
Proposition \ref{move points away from boundary} proves that for each $\textbf{z}_{I}$ one can construct a map $\phi_{R_{I}}:B_{R_{I}}(\textbf{z}_{I})\to F_{I}\big(K(n)\big)$ such that $\theta_{|I|}\big(\phi_{\textbf{z}_{I}}(\textbf{z}'_{I})\big)>\theta_{|I|}(\textbf{z}'_{I})$ for all $\textbf{z}'\in B_{R_{I}}(\textbf{z}_{I})$.
Here $R_{I}=\frac{1}{16}\min\{r, 1-r, r'_{I}-r, r''_{I}-r\}$, where $r'_{I}=\frac{1}{2}\min_{i\neq j\in I}\big\{d_{\infty}(z_{i}, z_{j})|d_{\infty}(z_{i}, z_{j})\neq 2r\big\}$, and $r''_{I}=\min_{i\in I}\Big\{d_{\infty}\big(z_{i}, \partial K(n)\big)|d_{\infty}\big(z_{i}, \partial K(n)\big)\neq r\Big\}$ when defined.
We use these $\phi_{R_{I}}$ to construct a map $\phi_{\textbf{z}}$ that increases $\theta_{m}$ on a neighborhood of $\textbf{z}$.

To do this, let $r'=\frac{1}{2}\min_{i\neq j}\big\{d_{\infty}(z'_{i}, z'_{j})|d_{\infty}(z'_{i}, z'_{j})\neq 2r\big\}$, and set $R=\min\Big\{\min_{I}\{R_{I}\}, \frac{1}{16}(r'-r)\Big\}$, and set $U_{z}=B_{R}(\textbf{z})$.
We claim that the map $\phi_{\textbf{z}}:U_{z}\to B_{2R}(\textbf{z})\subset F_{m}\big(K(n)\big)$ that restricts to $\phi_{R_{I}}$ on $\textbf{z}'_{I}$, increases $\theta_{m}$ for all $\textbf{z}'\in U_{\textbf{z}}$.
Since $B_{R}(\textbf{z}_{I})\subseteq B_{R_{I}}(\textbf{z}_{I})$ for all $I$, Proposition \ref{move points away from boundary} proves that $\phi_{\textbf{z}}$ is well-defined and smooth.
Moreover, it follows from Proposition and \ref{move points away from boundary} that if we also write $\phi_{\textbf{z}}(z'_{i})$ for the restriction of $\phi_{\textbf{z}}$ to a single point $z'_{i}$ in the configuration $\textbf{z}'\in U_{\textbf{z}}$, then $\phi_{\textbf{z}}$ increases $\frac{1}{2}d_{\infty}(z'_{i}, z'_{j})$ for all pairs $z'_{i}$ and $z'_{j}$ such that $z_{i}$ and $z_{j}$ are in the same component of $\Gamma_{\textbf{z}}$.
Similarly, if $K(n)=K_{g,1}(n)$, then Proposition \ref{move points away from boundary} shows that $\phi_{\textbf{z}}$ increases $d_{\infty}\big(z'_{i}, \partial K(n)\big)$ for all points $z'_{i}$ in the configuration $\textbf{z}'\in U_{\textbf{z}}$.
It remains to check that if $z_{i}$ and $z_{j}$ are not in the same component of $\Gamma_{\textbf{z}}$, then $\frac{1}{2}d_{\infty}\big(\phi_{\textbf{z}}(z'_{i}), \phi_{\textbf{z}}(z'_{j})\big)$ does not realize $\theta_{m}\big(\phi_{\textbf{z}}(\textbf{z}')\big)$ for all $\textbf{z}'\in U_{\textbf{z}}$.

Proposition \ref{move points away from boundary} proves that $\theta_{m}\big(\phi_{\textbf{z}}(\textbf{z}')\big)\le r+2R$ for all $\textbf{z}'\in U_{\textbf{z}}$.
Thus, it suffices to check that if $z_{i}$ and $z_{j}$ are not in the same component of $\Gamma_{\textbf{z}}$, then $\frac{1}{2}d_{\infty}\big(\phi_{\textbf{z}}(z'_{i}), \phi_{\textbf{z}}(z'_{j})\big)>r+2R$.
Note that $d_{\infty}\big(z'_{i}, \phi_{\textbf{z}}(z'_{i})\big)\le R$ for every point $z'_{i}$ in the configuration $\textbf{z}'$, since the $\phi_{R_{I}}$s move points a distance at most $R$.
Therefore, if $z_{i}$ and $z_{j}$ are not in the same component of $\Gamma_{\textbf{z}}$, then 
\begin{align*}
\frac{1}{2}d_{\infty}\big(\phi_{z}(z'_{i}),\phi_{z}(z'_{j})\big)&\ge -\frac{1}{2}d_{\infty}\big(\phi_{z}(z'_{i}),z'_{i}\big)+\frac{1}{2}d_{\infty}(z'_{i},z'_{j})-\frac{1}{2}d_{\infty}\big(z'_{j},\phi_{z}(z'_{j})\big)\\
&\ge -\frac{1}{2}R+r'-\frac{1}{2}R\\
&=r'-R\\
&>r+2R,
\end{align*}
where $z'_{i}$ and $ z'_{j}$ points of a configuration $\textbf{z}'\in U_{\textbf{z}}$.
Therefore, $\phi_{\textbf{z}}$ is a smooth function on $U_{\textbf{z}}=B_{R}(\textbf{z})$ that increases $\theta_{m}$ for all points in $U_{\textbf{z}}$ 
\end{proof}

Since the $\phi_{R_{I}}$s lead to a smooth vector field on the $B_{R}(\textbf{z}_{I})$s, the function $\phi_{\textbf{z}}$ leads to a smooth vector field on $U_{\textbf{z}}$.
Namely, at each point $\textbf{z}'\in U_{\textbf{z}}$ the map $\phi_{\textbf{z}}$ yields the vector $\vec{\phi}(\textbf{z}):=\hat{\phi}_{\textbf{z}}(\hat{\textbf{z}}')-\hat{\textbf{z}}'$, where $\hat{\textbf{z}'}$ is the point in $\big(S(n)\big)^{m}$ corresponding to $\textbf{z}'$ and $\hat{\phi}_{\textbf{z}}$ is the map inducing $\phi_{\textbf{z}}$ in $T_{\textbf{z}'}F_{m}\big(K(n)\big)$.
These vectors lie in $\text{Cone}_{\textbf{z}'}(\theta_{m})$, i.e., their dot products with the vectors coming from the contact graph $\Gamma_{\textbf{z}}$ are positive; moreover, for the reasons stated in Remark \ref{no need to worry about tangent space} we do not need to worry about our definition of the tangent space of $F_{m}\big(K_{g,0}(n)\big)$ as if a point in a configuration tends to $p$, the corresponding components of the vector tends to $0$.

As a result of Lemma \ref{increase tautological function},  we know that if $\theta_{m}(\textbf{z})<\frac{1}{2}$, then $\textbf{z}$ is a regular point, which proves the following corollary.

\begin{cor}\label{small regular values}
Let $K$ be a cube complex of the form $K_{g,0}(n)$ or $K_{g,1}(n)$.
Then, for $m\le n$, all $r\in \big(0,\frac{1}{2}\big)$ are regular values of $\theta_{m}:K^{m}\to \R$.
\end{cor}

Note that to prove Corollary \ref{small regular values} we constructed functions $\phi_{\textbf{z}}$ that increase $\theta_{m}$ on a neighborhood $U_{\textbf{z}}$ of $\textbf{z}$, and that this is equivalent to the existence of a smooth vector field on that neighborhood.
Let $V_{\textbf{z}}$ be a small open ball whose closure $\overline{V_{\textbf{z}}}$ is contained in $U_{\textbf{z}}$, and let $\phi_{t}$ denote the flow along the vector field given by $\phi_{\textbf{z}}$, which we call $\vec{\phi}_{\textbf{z}}$.
It follows that $\theta_{m}$ increases on $V_{\textbf{z}}$ along trajectories of $\phi_{t}$ with non-zero speed.
We use this to prove the following lemma, which is the first step of retracting $F_{m}\big(K(n)\big)$ onto $DF_{m}\big(K(n)\big)$.

\begin{lem}\label{deformation retract for non crit values}
Let $\theta_{m}:\big(K(n)\big)^{m}\to \R_{+}$ be the tautological function.
If $m\le n$ and $0<r_{1}<r_{2}<\frac{1}{2}$, the superlevel space $\theta^{-1}_{m}([r_{2}, \infty))$ is a deformation retract of the superlevel space $\theta^{-1}_{m}([r_{1}, \infty))$.
\end{lem}

Our proof of Lemma \ref{deformation retract for non crit values} is nearly identical to that of \cite[Theorem 16]{plachta2021configuration} and \cite[Lemma 3.2]{baryshnikov2014min}; we direct the reader to those works for further details.

\begin{proof}
Let $M$ be the subspace of $F_{m}\big(K(n)\big)$ consisting of configurations with tautological value between $r_{1}$ and $r_{2}$, i.e., $M:=\theta_{m}^{-1}[r_{1}, r_{2}]$.
By Lemma \ref{increase tautological function}, we know that for each $\textbf{z}\in F_{m}\big(K(n)\big)$ such that $\theta_{m}(\textbf{z})<\frac{1}{2}$ there is an open neighborhood $U_{\textbf{z}}$ of $\textbf{z}$ such that there exists a smooth vector field $\vec{\phi}_{\textbf{z}}$ on $U_{\textbf{z}}$ with $\vec{\phi}_{\textbf{z}}(\textbf{z}')\in \text{Cone}_{\textbf{z}'}(\theta_{m})$ for all $\textbf{z}'\in U_{\textbf{z}}$.
By restricting to an open subspace $V_{\textbf{z}}$ of $U_{\textbf{z}}$ such that $\overline{V_{\textbf{z}}}\subset U_{\textbf{z}}$, we know that flowing along this vector field for small $t$ keeps points that started relatively far apart, relatively far apart, while separating nearby points.

Note that $\mathcal{V}=\{V_{\textbf{z}}\}_{\textbf{z}\in M}$ is a cover of $M$.
Since $M$ is compact, we can find a subcover of $M$ by a finite subset $\mathcal{W}$ of $\mathcal{V}$.
Set $W(M)=\bigcup_{V_{\textbf{z}}\in\mathcal{W}}V_{\textbf{z}}$, and note that $W(M)$ is an open neighborhood of $M$ in $F_{m}\big(K(n)\big)$.
Taking a partition of unity subordinate to $\mathcal{W}$ we get a smooth vector field $\vec{\Phi}$ on $W(M)$.
Namely, let $\textbf{z}$ be a point of $M$, and let $V_{\textbf{z}_{\textbf{1}}},\dots, V_{\textbf{z}_{\textbf{k}}}$ be all the sets in $\mathcal{W}$ that contain $\textbf{z}$.
We set $\vec{\Phi}(\textbf{z})=\sum^{k}_{i=1}\lambda_{i}(\textbf{z})\vec{\phi}_{\textbf{z}_{\textbf{i}}}(\textbf{z})$, where the $\lambda_{i}$ are nonnegative functions that add to $1$ at $\textbf{z}$.
By the construction of the functions $\phi_{\textbf{z}_{\textbf{i}}}$ in Lemma \ref{increase tautological function}, the vectors $\vec{\phi}_{\textbf{z}_{\textbf{i}}}$ lie in $\text{Cone}_{\textbf{z}}(\theta_{m})$, so the sum $\vec{\Phi}(\textbf{z})$ must also lie in this cone, as all its summands satisfy the dot product requirements with the vectors arising from the contact graph of $\textbf{z}$.
Informally, since all of the functions yielding the $\vec{\phi}_{\textbf{z}_{\textbf{i}}}$ move nearby points farther apart, the function you get when you average these functions must also separate nearby points.
It follows that $\vec{\Phi}(\textbf{z})$ 
is a nowhere-zero smooth vector field in the neighborhood $W(M)$ of $M$.
Therefore, it generates a unique flow $\Phi_{t}$ on a neighborhood of $M$ in $F_{n}(K)$.
By construction, $\theta_{m}$ increases along the trajectories of the flow with non-zero speed, and $\Phi_{t}$ is transverse to $\theta^{-1}_{m}(r_{1})$ and $\theta^{-1}_{m}(r_{2})$, as each trajectory intersects these subspaces at a unique point.
Flowing along trajectories of $\Phi_{t}$ yields a deformation retraction of $\theta_{m}^{-1}[r_{1}, \infty)$ onto $\theta_{m}^{-1}[r_{2}, \infty)$.
\end{proof}

Next, we extend the deformation retract of Lemma \ref{deformation retract for non crit values} to $r_{2}=\frac{1}{2}$.
This shows that $SF_{m}\big(K(n), d\big)$ is homotopy equivalent to $SF_{m}\big(K(n), 1\big)$ for all $0<d\le 1$.

\begin{lem}\label{retract from unit square to small square}
Let $0<r_{1}<r_{2}<\frac{1}{2}$ be such that there is no critical value of $\theta_{m}$ in the interval $[r_{1}, r_{2})$.
Then, the superlevel space $\theta^{-1}_{m}([r_{2}, \infty))$ is a deformation retract of the superlevel space $\theta^{-1}_{m}([r_{1}, \infty))$.
\end{lem}

Plachta handles the case $K(n)=K_{0,1}(n)$ in \cite[Corollary 17]{plachta2021configuration}, using an argument that solely relies on the affine Morse--Bott nature of $\theta_{m}$ and a version of Lemma \ref{deformation retract for non crit values}.
Since we have proven analogous statements for general $K(n)$ in Proposition \ref{tautological function is affine Morse--Bott} and Lemma \ref{deformation retract for non crit values}, we omit the proof of Lemma \ref{retract from unit square to small square} and direct the reader to Plachta for the argument.

\begin{remark}
In the case $K(n)=K_{g,0}(n)$, Lemma \ref{retract from unit square to small square} is actually superfluous, as Corollary \ref{small regular values} can be strengthened.
Namely, our arguments in Propositions \ref{nearby points have similar contact graphs} and \ref{move points away from boundary} and Lemma \ref{increase tautological function}  can be extended to prove that all $r\in \big(0, \frac{n+1}{2n}\big)$ are regular values of $\theta_{m}$.
Since Lemma \ref{deformation retract for non crit values} can be extended in a similar manner, this is enough to prove Theorem \ref{point is square} in the $K(n)=K_{g,0}(n)$ setting.
\end{remark}

We use the previous lemma to prove that as long one does not have too many squares, the configuration space of open unit-squares in $K$ is homotopy equivalent to the ordered configuration space of points in $K$, noting that $\theta_{m}^{-1}\big([r, \infty)\big)$ can be identified with $SF_{m}\big(K(n), 2r\big)$.

\begin{T1}
  \thmtext
\end{T1} 

\begin{proof}\label{point is homotopy to square}
Let $\epsilon_{0}>\epsilon_{1}>\cdots>\epsilon_{i}>\cdots$ be a sequence of positive numbers $\epsilon_{i}$ such that $\lim_{i\to \infty}\epsilon_{i}=0$ and $\epsilon_{0}<1$.
We have $F_{m}\big(K(n)\big)=\bigcup_{i=0}^{\infty}SF_{m}\big(K(n), \epsilon_{i}\big)$.
Lemma \ref{retract from unit square to small square} proves that $SF_{m}\big(K(n), \epsilon_{i}\big)$ is a deformation retract of $SF_{m}\big(K(n), \epsilon_{i+1}\big)$ for all $i$; let $\rho_{i}$ be the corresponding deformation retraction.
Concatenating the $\rho_{i}$s yields a deformation retraction $\rho$ of $F_{m}\big(K(n)\big)$ onto the square configuration space $SF_{m}\big(K(n), \epsilon_{0}\big)$.
Combining this with the deformation retraction of $SF_{m}\big(K(n), \epsilon_{0}\big)=\theta_{m}^{-1}\big([\frac{\epsilon_{0}}{2}, \infty)\big)$ onto $\theta_{m}^{-1}\big([\frac{1}{2}, \infty)\big)=SF_{m}\big(K(n), 1\big)$ given by Lemma \ref{retract from unit square to small square} yields the homotopy equivalence of $F_{m}\big(K(n)\big)$ and $SF_{m}\big(K(n),d\big)$ for $d\le 1$.
The fact that configuration spaces are homotopy invariant up to homeomorphism of the base space yields the second homotopy equivalence.
\end{proof}

In the next section we prove that the configuration space of unit-squares in $K(n)$ is homotopy equivalent to a discrete configuration space arising from the cube complex structure of $K(n)$ in the $K_{g,0}(n)$ case and from $K^{*}(n)$ in the $K_{g,1}(n)$ case.
This equivalence gives a cubical model for the configuration space of points in a surface with up to one boundary component.

\section{Discrete Configuration Spaces}\label{discrete configuration spaces}

In this section we prove that the square configuration spaces of the previous section can be discretized in a particularly nice way, that is, we prove Theorem \ref{square is discrete}, which claims that $SF_{m}\big(K(n)\big)$ is homotopy equivalent to a certain cubical complex whose cells are labeled by products of cells in $K(n)$ or $K^{*}(n)$.
This was shown to be true for $K(n)=K_{0,1}(n)$ by Alpert, Bauer, Kahle, MacPherson, and Spendlove in \cite{alpert2023homology}.
To prove our theorem, we draw upon their techniques, while heavily relying on the ``polar coordinates'' of Lemmas \ref{section of small set no boundary} and \ref{section of small set with boundary}, and taking special care of the point $p\in K_{g,0}(n)$ for $g>1$, which greatly complicates the problem.
We begin by recalling the definition of a discrete configuration space.

\begin{defn}
Let $X$ be CW-complex.
The \emph{$m^{\text{th}}$-discrete ordered configuration space of $X$}, denoted \emph{$DF_{m}(X)$}, is the subcomplex of $X^{m}$ consisting of cells $\Sigma=\sigma_{1}\times \cdots\times \sigma_{m}$ whose constituent cells $\sigma_{i}$ have non-overlapping closures in $X$, i.e.,
\[
DF_{m}(X):=\left\{\sigma_{1}\times \cdots\times \sigma_{m}\in X^{m}|\overline{\sigma_{i}}\cap \overline{\sigma_{j}}=\emptyset\text{ if }i\neq j\right\}.
\]
The quotient of $DF_{m}(X)$ by the natural symmetric group action yields $CF_{m}(X)$ the \emph{$m^{\text{th}}$-discrete unordered configuration space of $X$}.
\end{defn}

We will be interested in the discrete configuration spaces $DF_{m}\big(K_{g,0}(n)\big)$ and $DF_{m}\big(K^{*}_{g,1}(n)\big)$, where $K^{*}_{g,1}(n)$ is the dual complex of $K_{g,1}(n)$.
Our interest in the former is obvious, the latter less so.
As such, we make the following observation.

\begin{prop}
The open $l^{\infty}$-ball of radius $\frac{1}{2}$ about $z\in K_{g,1}(n)$ has trivial intersection with the boundary of $K_{g,1}$ if and only if $z\in K^{*}_{g,1}(n)$.
\end{prop}

It follows that a configuration in $SF_{m}(K_{g,1}(n)\big)$ can be thought of as a point in $\big(K^{*}_{g,1}(n)\big)^{m}$.
We will show that if $m\le n$, the square configuration spaces $SF_{m}\big(K_{g,0}(n)\big)$ and $SF_{m}\big(K_{g,1}(n)\big)$ are homotopy equivalent to $DF_{m}\big(K_{g,0}(n)\big)$ and $DF_{m}\big(K^{*}_{g,1}(n)\big)$, respectively.
We begin by noting that $DF_{m}(K_{g,0}(n)\big)$ and $DF_{m}(K^{*}_{g,1}(n)\big)$ are subspaces of their respective square configuration spaces.

\begin{prop}\label{fully contained cell}
A cell $\Sigma=\sigma_{1}\times\cdots\times\sigma_{m}$ of $\big(K_{g,0}(n)\big)^{m}$ is fully contained in contained in $SF_{m}\big(K_{g,0}(n)\big)$ if and only if $\Sigma$ is a cell of $DF_{m}\big(K_{g,0}(n)\big)$.
Similarly, a cell $\Sigma=\sigma_{1}\times\cdots\times\sigma_{m}$ of $\big(K^{*}_{g,1}(n)\big)^{m}$ is fully contained in  contained in $SF_{m}\big(K_{g,1}(n)\big)$ if and only if $\Sigma$ is a cell of $DF_{m}\big(K^{*}_{g,1}(n)\big)$.
\end{prop}

\begin{proof}
If $\Sigma$ is in $DF_{m}\big(K_{g,0}(n)\big)$, resp. $DF_{m}\big(K^{*}_{g,1}(n)\big)$, then $\overline{\sigma_{i}}\cap \overline{\sigma_{j}}=\emptyset$ for all $i\neq j$, so $d_{\infty}(\overline{\sigma_{i}},\overline{\sigma_{j}})\ge 1$ for all $i\neq j$.
Therefore, for every $z_{i}\in \overline{\sigma_{i}}$ and $z_{j}\in \overline{\sigma_{j}}$, the open unit-diameter $l^{\infty}$-balls around $z_{i}$ and $z_{j}$, i.e., the open unit-squares centered at $z_{i}$ and $z_{j}$, do not intersect.
It follows that such a $\Sigma$ is fully contained in $SF_{m}\big(K_{g,0}(n)\big)$, resp. $SF_{m}\big(K_{g,1}(n)\big)$.

Similarly, if $\Sigma$ is fully contained in $SF_{m}\big(K_{g,0}(n)\big)$, resp. $SF_{m}\big(K_{g,1}(n)\big)$, then for $i\neq j$ and all $z_{i}\in \overline{\sigma_{i}}$ and $z_{j}\in \overline{\sigma_{j}}$, the open unit-diameter $l^{\infty}$-balls around $z_{i}$ and $z_{j}$ do not intersect.
It follows that $d_{\infty}(\overline{\sigma_{i}},\overline{\sigma_{j}})\ge 1$, so $\overline{\sigma_{i}}\cap \overline{\sigma_{j}}=\emptyset$ for all $i\neq j$ and $\Sigma$ is a cell of $DF_{m}\big(K_{g,0}(n)\big)$, resp. $DF_{m}\big(K^{*}_{g,1}(n)\big)$.
\end{proof}

In order to retract $SF_{m}\big(K_{g,0}(n)\big)$ onto $DF_{m}\big(K_{g,0}(n)\big)$ and $SF_{m}\big(K^{*}_{g,1}(n)\big)$ onto $DF_{m}\big(K^{*}_{g,1}(n)\big)$, we need to handle the cells of $\big(K_{g,0}(n)\big)^{m}$ and $\big(K^{*}_{g,1}(n)\big)^{m}$ that are only partially contained in $SF_{m}\big(K_{g,0}(n)\big)$ and $SF_{m}\big(K^{*}_{g,1}(n)\big)$.
These partially contained cells $\Sigma=\sigma_{1}\times\cdots\times \sigma_{m}$ are such that $\overline{\sigma_{i}}\cap\overline{\sigma_{j}}\neq \emptyset$ for some $i\neq j$, though there exists a point $\textbf{z}$ in $\Sigma$ that is a configuration in  $SF_{m}\big(K(n)\big)$.
We claim that the barycenter of a partially contained cell is such a square configuration.

\begin{prop}\label{barycenter in square}
Let $\Sigma=\sigma_{1}\times\cdots\times \sigma_{m}$ be a cell of $\big(K_{g,0}(n)\big)^{m}$, resp. $\big(K^{*}_{g,1}(n)\big)^{m}$, that is partially contained in $SF_{m}\big(K_{g,0}(n)\big)$, resp. $SF_{m}\big(K_{g,1}(n)\big)$, and let $\textbf{m}$ be its barycenter.
The configuration $\textbf{m}$ lies in $SF_{m}\big(K_{g,0}(n)\big)$, resp. $SF_{m}\big(K_{g,1}(n)\big)$.
\end{prop}

\begin{proof}
If $\Sigma=\sigma_{1}\times\cdots\times \sigma_{m}$ is partially contained in $SF_{m}\big(K(n)\big)$, then Proposition \ref{fully contained cell} implies that $\overline{\sigma_{i}}\cap\overline{\sigma_{j}}\neq \emptyset$ for some $i\neq j$.
Note that neither $\sigma_{i}$ nor $\sigma_{j}$ can be a $0$-cell as if this were the case, the Chebyshev distance between any pair of points in $\sigma_{i}$ and $\sigma_{j}$ would be less than $1$, so $\Sigma$ would have trivial intersection with $SF_{m}\big(K(n)\big)$.
Thus, we only need to consider cells $\Sigma=\sigma_{1}\times\cdots\times \sigma_{m}$ where neither cell in a pair $\sigma_{i}, \sigma_{j}$ whose closures have non-trivial intersection is a $0$-cell.

In Figure \ref{fig:cellsofdistancezero} we depict the eight possible arrangements of such $\overline{\sigma_{i}}$ and $\overline{\sigma_{j}}$.
If $K(n)=K_{g,0}(n)$ and the intersection of the first type of pair is at $p$, then $\sigma_{i}$ and $\sigma_{j}$ are parallel, i.e., either both vertical or both horizontal.
Similarly, if $K(n)=K_{g,0}(n)$ and the intersection of the seventh type of pair is at $p$, then $\sigma_{i}$ and $\sigma_{j}$ are perpendicular, i.e., one is horizontal and the other is vertical.
If $\overline{\sigma_{i}}$ and $\overline{\sigma_{j}}$ have non-trivial intersection and $\sigma_{i}$ and $\sigma_{j}$ are any of the second four forms depicted in Figure \ref{fig:cellsofdistancezero}, then $\Sigma$ has trivial intersection with $SF_{m}\big(K(n)\big)$ as $d_{\infty}(z_{i}, z_{j})<1$ for all $z_{i}\in \sigma_{i}$ and $z_{j}\in \sigma_{j}$.
On the other hand, if every pair $\sigma_{i}$ and $\sigma_{j}$ such that $\overline{\sigma_{i}}\cap \overline{\sigma_{j}}\neq \emptyset$ is of one of the first four forms depicted in Figure \ref{fig:cellsofdistancezero}, then $\Sigma$ has non-trivial intersection with $SF_{m}\big(K(n)\big)$ as the Chebyshev distance between the barycenters of such cells is $1$.
It follows that if $\Sigma=\sigma_{1}\times\cdots\times \sigma_{m}$ is partially contained in $SF_{m}\big(K(n)\big)$, then every pair of $\sigma_{i}$ and $\sigma_{j}$ whose closures have non-trivial intersection is of one of the first four forms depicted in Figure \ref{fig:cellsofdistancezero}, and the barycenter $\textbf{m}$ of $\Sigma$ is in $SF_{m}\big(K(n)\big)$.
\begin{figure}[h]
    \centering
    \includegraphics[width=.75\linewidth]{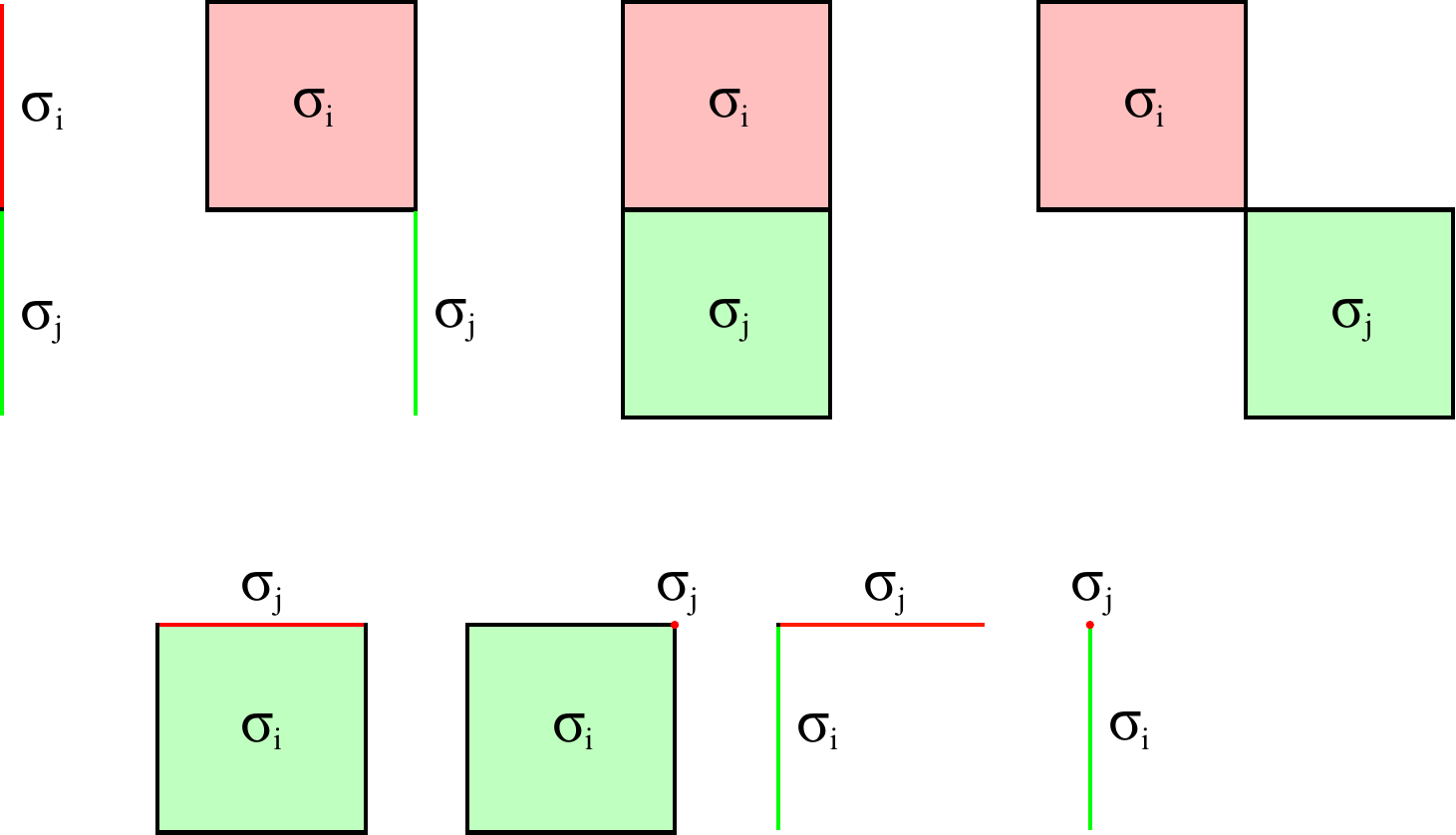}
    \caption{
    Distinct cells in $K_{g,0}(n)$ and $K^{*}_{g,1}(n)$ whose closures have non-trivial intersection. 
If $\sigma_{i}$ and $\sigma_{j}$ are of the four forms depicted in the first row, then there exist two disjoint open unit-squares, with one centered in $\sigma_{i}$ and the other in $\sigma_{j}$.
There are no such disjoint squares if $\sigma_{i}$ and $\sigma_{j}$ are of the four forms depicted in the second row.
    }
\label{fig:cellsofdistancezero}
\end{figure}
\end{proof}

Given a cell $\Sigma=\sigma_{1}\times\cdots\times \sigma_{m}$ partially contained in $SF_{m}\big(K(n)\big)$, we would like to determine which points $\textbf{z}$ of $\Sigma$ are configurations in $SF_{m}\big(K(n)\big)$.
We do this by finding inequalities arising from the points constituting $\textbf{z}$; to find such inequalities we note the existence of a particularly nice map from $\Sigma$ to $\R^{2m}$.

\begin{prop}\label{well-defined lift}
For $m\le n$, let $\Sigma=\sigma_{1}\times \cdots\times \sigma_{m}$ be a cell in $(K_{g,0}(n)\big)^{m}$, resp. $(K^{*}_{g,1}(n)\big)^{m}$, that is at least partially contained in $SF_{m}\big(K_{g,0}(n)\big)$, resp. $SF_{m}\big(K_{g,1}(n)\big)$.
There is a lift of each component of $\bigcup_{i=1}^{m}\overline{\sigma_{i}}$ to $\tilde{K}_{g,0}(n)$, resp. $\tilde{K}^{*}_{g,1}(n)$.
\end{prop}

\begin{proof}
Consider the connected components of $\bigcup_{i=1}^{m}\overline{\sigma_{i}}$.
Every path in each component can be homotoped relative its endpoints to a path of length at most $m\le n$; as such, it follows from Proposition \ref{injectivity radius} that each component of $\bigcup_{i=1}^{m}\overline{\sigma_{i}}$ is contractible in $K_{g,0}(n)$, resp. is homotopic into the boundary of $K^{*}_{g,1}(n)$.
Therefore, Lemma \ref{section of small set no boundary}, resp. Lemma \ref{section of small set with boundary}, proves that there is a section of $\pi:\tilde{K}_{g,0}(n)\to K_{g,0}(n)$, resp. $\pi:\tilde{K}^{*}_{g,1}(n)\to K^{*}_{g,1}(n)$, on each component.
This yields the desired lift. 
\end{proof}

Note that there are projection maps from $\tilde{K}_{g,0}(n)$ and $\tilde{K}^{*}_{g,1}(n)$ to $\R^{2}$ that arise from the definition of $\tilde{K}_{g,0}(n)$ as a quotient of the slit squares $S(n)$.
If $\tilde{z}$ is a point in $\tilde{K}_{g,0}(n)$, resp. $\tilde{K}^{*}_{g,1}(n)$, we let $\hat{z}$ denote the image of this point in $S(n)$, resp. $A^{*}(n)$, in $\R^{2}$.
With this in mind, if $\textbf{z}$ is a point in $\big(K(n)\big)^{m}$, and $\tilde{\textbf{z}}$ is one of its lifts to $\big(\tilde{K}(n)\big)^{m}$, then we write $\hat{\textbf{z}}$ for the image of this point in $\big(S(n)\big)^{m}\subset \R^{2m}$.
Note that $\textbf{z}$ need not be a configuration, and even it if is, $\hat{\textbf{z}}$ might not be.
Moreover, given a $\textbf{z}$ there is not a unique $\hat{\textbf{z}}$.
As will soon see, this is not an issue as the types of inequalities that arise do not depend on the choice of our lifts.

We use this map to find inequalities in $\R^{2}$ that configurations of two points $z_{1}$ and $z_{2}$ must satisfy in order for $\textbf{z}=(z_{1}, z_{2})$ to be contained in $SF_{2}\big(K(n)\big)$.
Before we do that, we introduce some notation.
Given a cell $\sigma$ in $K_{g,0}(n)$ or $K^{*}_{g,1}(n)$ and its image $\hat{\sigma}$ in $\R^{2}$ under the projection of a lift, let $\big(\text{bary}_{x}(\hat{\sigma}), \text{bary}_{y}(\hat{\sigma})\big)$ denote the coordinates of the barycenter of $\hat{\sigma}$.

\begin{prop}\label{local coordinate check boundary}
Let $\Sigma=\sigma_{1}\times\sigma_{2}$ be an open cell of $\big(K^{*}_{g,1}(n)\big)^{2}$ that is partially contained in $SF_{2}\big(K_{g,1}(n)\big)$, and let $\textbf{z}=(x_{1}, y_{1}; x_{2}, y_{2})$ be a point in $\Sigma$.
Then $\textbf{z}\in SF_{2}\big(K_{g,1}(n)\big)$ if and only if at least one of the following two conditions holds:
\begin{enumerate}
    \item $|\text{bary}_{x}(\hat{\sigma}_{2})-\text{bary}_{x}(\hat{\sigma}_{1})|=1$ and $\big(\hat{x}_{2}-\text{bary}_{x}(\hat{\sigma}_{2})\big)-\big(\hat{x}_{1}-\text{bary}_{x}(\hat{\sigma}_{1})\big)$ is $0$ or has the same sign as $\text{bary}_{x}(\hat{\sigma}_{2})-\text{bary}_{x}(\hat{\sigma}_{1})$, or

    \item $|\text{bary}_{y}(\hat{\sigma}_{2})-\text{bary}_{y}(\hat{\sigma}_{1})|=1$ and $\big(\hat{y}_{2}-\text{bary}_{y}(\hat{\sigma}_{2})\big)-\big(\hat{y}_{1}-\text{bary}_{y}(\hat{\sigma}_{1})\big)$ is $0$ or has the same sign as $\text{bary}_{y}(\hat{\sigma}_{2})-\text{bary}_{y}(\hat{\sigma}_{1})$,

\end{enumerate}

Similarly, if $\Sigma=\sigma_{1}\times\sigma_{2}$ be an open cell of $\big(K_{g,0}(n)\big)^{2}$ that is partially contained in $SF_{2}\big(K_{g,0}(n)\big)$, and $\textbf{z}=(x_{1}, y_{1}; x_{2}, y_{2})$ is a point in $\Sigma$.
Then $\textbf{z}\in SF_{2}\big(K_{g,0}(n)\big)$ if and only if at least one of the following four conditions holds:
\begin{enumerate}
   \item $|\text{bary}_{x}(\hat{\sigma}_{2})-\text{bary}_{x}(\hat{\sigma}_{1})|=1$ and $\big(\hat{x}_{2}-\text{bary}_{x}(\hat{\sigma}_{2})\big)-\big(\hat{x}_{1}-\text{bary}_{x}(\hat{\sigma}_{1})\big)$ is $0$ or has the same sign as $\text{bary}_{x}(\hat{\sigma}_{2})-\text{bary}_{x}(\hat{\sigma}_{1})$, 

    \item $|\text{bary}_{y}(\hat{\sigma}_{2})-\text{bary}_{y}(\hat{\sigma}_{1})|=1$ and $\big(\hat{y}_{2}-\text{bary}_{y}(\hat{\sigma}_{2})\big)-\big(\hat{y}_{1}-\text{bary}_{y}(\hat{\sigma}_{1})\big)$ is $0$ or has the same sign as $\text{bary}_{y}(\hat{\sigma}_{2})-\text{bary}_{y}(\hat{\sigma}_{1})$,

    \item $|\hat{x}_{1}|, |\hat{x}_{2}|, |\hat{y}_{1}|, |\hat{y}_{2}|<1$, $|\text{bary}_{x}(\hat{\sigma}_{2})-\text{bary}_{x}(\hat{\sigma}_{1})|=0$, and $|\hat{x}_{2}+\hat{x}_{1}|\ge 1$, or

     \item $|\hat{x}_{1}|, |\hat{x}_{2}|, |\hat{y}_{1}|, |\hat{y}_{2}|<1$, $|\text{bary}_{y}(\hat{\sigma}_{2})-\text{bary}_{y}(\hat{\sigma}_{1})|=0$, and $|\hat{y}_{2}+\hat{y}_{1}|\ge 1$.
\end{enumerate}
\end{prop}

\begin{proof}
We will see that the third and fourth conditions in the $\big(K_{g,0}(n)\big)^{2}$ case only pertain to cells $\Sigma=\sigma_{1}\times \sigma_{2}$ such that $p$, the point of non-zero curvature, is in $\overline{\sigma_{1}}\cap\overline{\sigma_{2}}$.
Since $K^{*}_{g,1}(n)$ is the subcomplex of $K_{g,0}(n)$ consisting of all cells $\sigma$ whose closures do not contain $p$, it follows that it suffices to prove this proposition in the case where $\Sigma$ is an open cell of $\big(K_{g,0}(n)\big)^{2}$.

The point $\textbf{z}=(z_{1}, z_{2})$ in $\Sigma=\sigma_{1}\times \sigma_{2}\subset \big(K_{g,0}(n)\big)^{2}$ is in $SF_{2}\big(K_{g,0}(n)\big)$ if and only if at least one of the horizontal and vertical distances between $z_{1}$ and $z_{2}$ is at least $1$.
Since $\Sigma=\sigma_{1}\times \sigma_{2}$ is only partially contained in $SF_{2}\big(K_{g,0}(n)\big)$, it must be the case that $\overline{\sigma_{1}}\cap \overline{\sigma_{2}}\neq \emptyset$, and by Proposition \ref{well-defined lift} there is a lift $\overline{\tilde{\sigma}_{1}}\cup\overline{\tilde{\sigma}_{2}}$ of $\overline{\sigma_{1}}\cup\overline{\sigma_{2}}$ in $\tilde{K}_{g,0}(n)$.
This leads to a lift $\tilde{\textbf{z}}$ of $\textbf{z}$, and we can consider the image $\hat{\textbf{z}}$ of $\tilde{\textbf{z}}$ under the projection map from $\tilde{K}_{g,0}(n)$ to $S(n)$.
With this in mind, we wish to translate the condition that $d_{\infty}(z_{1}, z_{2})\ge 1$ into inequalities involving their images $\hat{z}_{1}$ and $\hat{z}_{2}$ in $\R^{2}$.

Note that the map from $\overline{\sigma_{1}}\cup\overline{\sigma_{2}}$ to $\overline{\tilde{\sigma}_{1}}\cup\overline{\tilde{\sigma}_{2}}$ is an isometry, preserving both the horizontal and vertical distances between pairs of points.
Moreover, if $0\notin\overline{\tilde{\sigma}_{1}}\cap\overline{\tilde{\sigma}_{2}}$, then  $\overline{\tilde{\sigma}_{1}}\cup\overline{\tilde{\sigma}_{2}}$ is horizontally and vertically isometric to $\overline{\hat{\sigma}_{1}}\cup\overline{\hat{\sigma}_{2}}$.
It follows that if $p\notin\overline{\sigma_{1}}\cap\overline{\sigma_{2}}$, then  $\overline{\sigma_{1}}\cup\overline{\sigma_{2}}$ is horizontally and vertically isometric to $\overline{\hat{\sigma}_{1}}\cup\overline{\hat{\sigma}_{2}}$, so $\hat{\sigma}_{1}\cup\hat{\sigma}_{2}$ is of one of the first four forms depicted in Figure \ref{fig:cellsofdistancezero}.
Therefore, in this case, if $\hat{\textbf{z}}=(\hat{z}_{1}, \hat{z}_{2})=(\hat{x}_{1}, \hat{y}_{1};\hat{x}_{2}, \hat{y}_{2})$ is the image of $\textbf{z}$ under this map, then $\textbf{z}$ is in $SF_{2}\big(K_{g,0}(n)\big)$ if and only if either $|\hat{x}_{2}-\hat{x}_{1}|\ge 1$ or $|\hat{y}_{2}-\hat{y}_{1}|\ge 1$.

If $p\in\overline{\sigma_{1}}\cap\overline{\sigma_{2}}$, then the map that sends $\overline{\tilde{\sigma}_{1}}\cup\overline{\tilde{\sigma}_{2}}$ to $\overline{\hat{\sigma}_{1}}\cap\overline{\hat{\sigma}_{2}}$ need not be an isometry, so the resulting map from $\overline{\tilde{\sigma}_{1}}\cup\overline{\tilde{\sigma}_{2}}$ to $\overline{\hat{\sigma}_{1}}\cap\overline{\hat{\sigma}_{2}}$ need not be an isometry, and at least one of the horizontal and vertical distances might be distorted.
That said, the restriction of this map to either $\overline{\sigma_{1}}$ or $\overline{\sigma_{2}}$ does preserve these distances.
We consider how this affects the induced maps from $\Sigma$ to $\R^{4}$.

If $\sigma_{1}$ and $\sigma_{2}$ are both $1$-cells, then, since $\Sigma$ is partially contained in $SF_{2}\big(K_{g,0}(n)\big)$, they are parallel, i.e., both vertical or both horizontal, i.e., of the first form in Figure \ref{fig:cellsofdistancezero}.
In this case, $\hat{\sigma}_{1}\cup \hat{\sigma}_{2}$ consists of either two disjoint parallel $1$-cells or a single $1$-cell; ether way, we may assume without loss of generality that both lines are vertical.
If $\hat{\sigma}_{1}\cup \hat{\sigma}_{2}$ consists of two disjoint parallel $1$-cells, then $\overline{\sigma_{1}}\cup\overline{\sigma_{2}}$ is isometric to $\overline{\hat{\sigma}_{1}}\cup\overline{\hat{\sigma}_{2}}$ and $\hat{\textbf{z}}=(\hat{z}_{1}, \hat{z}_{2})=(\hat{x}_{1}, \hat{y}_{1};\hat{x}_{2}, \hat{y}_{2})=(0, \hat{y}_{1};0, \hat{y}_{2})$.
It follows that $\textbf{z}$ is in $SF_{2}\big(K_{g,0}(n)\big)$ if and only if $|\hat{y}_{2}-\hat{y}_{1}|\ge 1$. 
Otherwise, $\hat{\sigma}_{1}\cup \hat{\sigma}_{2}$ consists of a single $1$-cell, and $\textbf{z}$ is in $SF_{2}\big(K_{g,0}(n)\big)$ if and only if $\hat{\textbf{z}}=(\hat{z}_{1}, \hat{z}_{2})=(\hat{x}_{1}, \hat{y}_{1};\hat{x}_{2}, \hat{y}_{2})=(0, \hat{y}_{1};0, \hat{y}_{2})$ is such that $|\hat{y}_{2}+\hat{y}_{1}|\ge 1$. 

If $\sigma_{1}$ is a $1$-cell and $\sigma_{2}$ is a $2$-cell, then there are two possible cases for the image $\overline{\hat{\sigma}_{1}}\cup\overline{\hat{\sigma}_{2}}$ of $\overline{\sigma_{1}}\cup\overline{\sigma_{2}}$ in $\R^{2}$.
In the first case, $\overline{\hat{\sigma}_{1}}\cup\overline{\hat{\sigma}_{2}}$ is of the second form in Figure \ref{fig:cellsofdistancezero}.
It follows that $\overline{\hat{\sigma}_{1}}\cup\overline{\hat{\sigma}_{2}}$ and $\overline{\sigma_{1}}\cup\overline{\sigma_{2}}$ are isometric, so $\hat{\textbf{z}}=(\hat{z}_{1}, \hat{z}_{2})=(\hat{x}_{1}, \hat{y}_{1};\hat{x}_{2}, \hat{y}_{2})=(0, \hat{y}_{1};\hat{x}_{2}, \hat{y}_{2})$.
Therefore, $\textbf{z}$ is in $SF_{2}\big(K_{g,0}(n)\big)$ if and only if $|\hat{y}_{2}-\hat{y}_{1}|\ge 1$. 
In the other case, $\overline{\hat{\sigma}_{1}}\cup\overline{\hat{\sigma}_{2}}$ is of the fifth form in Figure \ref{fig:cellsofdistancezero}, and $\textbf{z}$ is in $SF_{2}\big(K_{g,0}(n)\big)$ if and only if $\hat{\textbf{z}}=(\hat{z}_{1}, \hat{z}_{2})=(\hat{x}_{1}, \hat{y}_{1};\hat{x}_{2}, \hat{y}_{2})=(0, \hat{y}_{1};\hat{x}_{2}, \hat{y}_{2})$ is such that $|\hat{y}_{2}+\hat{y}_{1}|\ge 1$, i.e., the fourth condition holds since the origin is contained in $\overline{\hat{\sigma}_{1}}\cap\overline{\hat{\sigma}_{2}}$.

Finally, we need to consider the case in which $\sigma_{1}$ and $\sigma_{2}$ are both $2$-cells.
If $\sigma_{1}$ and $\sigma_{2}$ are of the third form depicted in Figure \ref{fig:cellsofdistancezero} (up to rotation), then the map sending $\overline{\sigma_{1}}\cup\overline{\sigma_{2}}$ to $\overline{\hat{\sigma}_{1}}\cup\overline{\hat{\sigma}_{2}}$ is an isometry.
Then, without loss of generality, a configuration $\textbf{z}$ with a point in each of $\sigma_{1}$ and $\sigma_{2}$ is in $SF_{2}\big(K_{g,0}(n)\big)$ if and only if $\hat{\textbf{z}}=(\hat{z}_{1}, \hat{z}_{2})=(\hat{x}_{1}, \hat{y}_{1};\hat{x}_{2}, \hat{y}_{2})$ is such that $|\hat{x}_{2}-\hat{x}_{1}|\ge 1$ as the vertical distance between $\hat{z}_{1}$ and $\hat{z}_{2}$ is less than $1$.
Otherwise, $\sigma_{1}$ and $\sigma_{2}$ are of the fourth form depicted in Figure \ref{fig:cellsofdistancezero}.
In this case, the map from $\overline{\sigma_{1}}\cup\overline{\sigma_{2}}$ to $\overline{\hat{\sigma}_{1}}\cup\overline{\hat{\sigma}_{2}}$ need not be a horizontal or vertical isometry, though it might be.
If it is, then $\textbf{z}\in SF_{2}\big(K_{g,0}(n)\big)$ if and only if $\hat{\textbf{z}}=(\hat{z}_{1}, \hat{z}_{2})=(\hat{x}_{1}, \hat{y}_{1};\hat{x}_{2}, \hat{y}_{2})$ is such that at least one of  $|\hat{x}_{2}-\hat{x}_{1}|\ge 1$ and $|\hat{y}_{2}-\hat{y}_{1}|\ge 1$, i.e., one of the first two conditions hold.
If the map from $\overline{\sigma_{1}}\cup\overline{\sigma_{2}}$ to $\overline{\hat{\sigma}_{1}}\cup\overline{\hat{\sigma}_{2}}$ maps  $\overline{\sigma_{1}}\cup\overline{\sigma_{2}}$ to the third form in Figure \ref{fig:cellsofdistancezero}, then $\textbf{z}\in SF_{2}\big(K_{g,0}(n)\big)$ if and only if $\hat{\textbf{z}}=(\hat{z}_{1}, \hat{z}_{2})=(\hat{x}_{1}, \hat{y}_{1};\hat{x}_{2}, \hat{y}_{2})$ is such that at least one of  $|\hat{x}_{2}+\hat{x}_{1}|\ge 1$ and $|\hat{y}_{2}-\hat{y}_{1}|\ge 1$, as the horizontal distance between $z_{1}$ and $z_{2}$ is $|\hat{x}_{2}|+|\hat{x}_{1}|=|\hat{x}_{2}+\hat{x}_{1}|$ and the vertical distance between them is $|\hat{y}_{2}-\hat{y}_{1}|$.
Finally, we must consider the case in which $\overline{\sigma_{1}}\cup\overline{\sigma_{2}}$ is sent to a single $2$-cell.
In this case, $\textbf{z}$ is in $SF_{2}\big(K_{g,0}(n)\big)$ if and only if $\hat{\textbf{z}}=(\hat{z}_{1}, \hat{z}_{2})=(\hat{x}_{1}, \hat{y}_{1};\hat{x}_{2}, \hat{y}_{2})$ is such that at least one of  $|\hat{x}_{2}+\hat{x}_{1}|\ge 1$ and $|\hat{y}_{2}+\hat{y}_{1}|\ge 1$ as the vertical distance between $z_{1}$ and $z_{2}$ is $|\hat{y}_{2}|+|\hat{y}_{1}|=|\hat{y}_{2}+\hat{y}_{1}|$ and the horizontal distance between them is $|\hat{x}_{2}|+|\hat{x}_{1}|=|\hat{x}_{2}+\hat{x}_{1}|$. 

It remains to check that these inequalities correspond to the four conditions in the proposition statement.
Note that the function $\text{bary}_{x}:\R^{2} \to \R$ that sends $\hat{z}$ to the $x$-coordinate of the barycenter of the open cell $\hat{\sigma}$ in $\R^{2}$ containing it is weakly-order preserving, i.e., if $\hat{z}_{1}=(\hat{x}_{1}, \hat{y}_{1})$ and $\hat{z}_{2}=(\hat{x}_{2}, \hat{y}_{2})$ are such that $\hat{x}_{1}>\hat{x}_{2}$, then $\text{bary}_{x}(\hat{z}_{1})\ge \text{bary}_{x}(\hat{z}_{2})$.
By the symmetry of $\hat{x}_{1}$ and $\hat{x}_{2}$ as well as $(\hat{x}_{1}, \hat{x}_{2})$ and $(\hat{y}_{1}, \hat{y}_{2})$, it suffices to note that $|\hat{x}_{2}-\hat{x}_{1}|\ge 1$ if and only if both $\text{bary}_{x}(\hat{\sigma}_{2})-\text{bary}_{x}(\hat{\sigma}_{2})=1$ and $\big(\hat{x}_{2}-\text{bary}_{x}(\hat{\sigma}_{2})\big)-\big(\hat{x}_{1}-\text{bary}_{x}(\hat{\sigma}_{1})\big)\ge 0$.
Similarly, we only get the inequality of the form $|\hat{x}_{2}+\hat{x}_{2}|\ge 1$ if $p\in \overline{\sigma_{1}}\cap\overline{\sigma_{2}}$, in which case $0\in \overline{\hat{\sigma}_{1}}\cap\overline{\hat{\sigma}_{2}}$, i.e., $|\hat{x}_{1}|, |\hat{x}_{2}|, |\hat{y}_{1}|, |\hat{y}_{2}|<1$.
Moreover, in this case, this inequality arises if and only if $|\text{bary}_{x}(\hat{\sigma}_{2})-\text{bary}_{x}(\hat{\sigma}_{1})|=0$, completing the cataloging of conditions.
\end{proof}

Note that if we write $\hat{z}=(\hat{x}, \hat{y})$ in terms of local coordinates $(\hat{u}, \hat{v}):=\big(\hat{x}-\text{bary}_{x}(\hat{\sigma}),\hat{y}-\text{bary}_{y}(\hat{\sigma})\big)$, the first condition in Proposition \ref{local coordinate check boundary} translates to an inequality of the form $\hat{u}_{1}\ge \hat{u}_{2}$ or $\hat{u}_{1}\le \hat{u}_{2}$, and a similar statement holds for the second condition.

We use these inequalities to find a deformation retract of the parts of $\Sigma$ contained in $SF_{m}\big(K(n)\big)$ onto the intersection of the boundary of $\Sigma$ with $SF_{m}\big(K(n)\big)$.
Doing so in sequence for all partially contained cells will yield a deformation retract of $SF_{m}\big(K(n)\big)$ onto our discrete configuration space.

\begin{lem}\label{retract partially contained cell no boundary}
For $m\le n$, let $\Sigma=\sigma_{1}\times \cdots\times \sigma_{m}$ be an open cell in $\big(K_{g,0}(n)\big)^{m}$, resp. $\big(K^{*}_{g,1}(n)\big)^{m}$, that is partially contained in $SF_{m}\big(K_{g,0}(n)\big)$, resp. $SF_{m}\big(K_{g,1}(n)\big)$.
Then, $\overline{\Sigma}\cap SF_{m}\big(K_{g,0}(n)\big)$, resp. $\overline{\Sigma}\cap SF_{m}\big(K_{g,1}(n)\big)$, deformation retracts onto $\partial\overline{\Sigma}\cap SF_{m}\big(K_{g,0}(n)\big)$, resp. $\partial\overline{\Sigma}\cap SF_{m}\big(K_{g,1}(n)\big)$.
\end{lem}

\begin{proof}
Note that as $K^{*}_{g,1}(n)$ is the subcomplex of $K_{g,0}(n)$ consisting of all cells $\sigma$ whose closure does not contain $p$, the point of non-zero curvature, it suffices to prove this lemma in the $SF_{m}\big(K_{g,0}(n)\big)$ setting.

Let $\textbf{z}$ be a point in $\Sigma$.
Proposition \ref{local coordinate check boundary} gives conditions that can be used to determine whether $\textbf{z}$ is in $SF_{m}\big(K_{g,0}(n)\big)$.
Namely, if $\tilde{\textbf{z}}$ is a lift of $\textbf{z}$ in $\big(\tilde{K}_{g,0}(n)\big)^{m}$ arising from the lifts of the connected components of $\bigcup_{i=1}^{m}\overline{\sigma_{i}}$ to $\tilde{K}^{*}_{g,0}(n)$ guaranteed by Proposition \ref{well-defined lift}, and $\hat{\textbf{z}}$ is the projection of $\tilde{\textbf{z}}$ in $\R^{2m}$ arising from the natural projection of $\tilde{K}_{g,0}(n)$ to $\R^{2}$, then Proposition \ref{local coordinate check boundary} can be interpreted as giving a set of inequalities on the (local) coordinates of $\hat{\textbf{z}}=(\hat{x}_{1}, \hat{y}_{1};\dots;\hat{x}_{m}, \hat{y}_{m})$ that we can use to determine if $\textbf{z}$ is in $SF_{m}\big(K_{g,0}(n)\big)$.
That is, write $\hat{\textbf{z}}-\text{bary}(\hat{\textbf{z}})=(\hat{u}_{1}, \hat{v}_{1};\dots;\hat{u}_{m}, \hat{v}_{m})\subset \big(-\frac{1}{2},\frac{1}{2}\big)^{2m}$ for the local coordinates of $\hat{\textbf{z}}$ within the cell $\hat{\Sigma}$.
Namely, for each pair $z_{k},z_{l}$ of points in $\textbf{z}$, Proposition \ref{local coordinate check boundary} specifies zero, one, or two inequalities of the following forms:
\begin{enumerate}
    \item $\hat{u}_{k}\ge \hat{u}_{l}$ or $\hat{u}_{k}\le \hat{u}_{l}$;
    \item $\hat{v}_{k}\ge \hat{v}_{l}$ or $\hat{v}_{k}\le \hat{v}_{l}$;
    \item $|\hat{x}_{k}+\hat{x}_{l}|\ge 1$;
    \item $|\hat{y}_{k}+\hat{y}_{l}|\ge 1$.
\end{enumerate}

There are zero inequalities if $\overline{\sigma_{k}}$ and $\overline{\sigma_{l}}$ do not intersect.
There is an inequality of type 1, resp. 2, if $\overline{\sigma_{k}}$ and $\overline{\sigma_{l}}$ are of any of the first, second, or third forms seen in Figure \ref{fig:cellsofdistancezero} (up to rotation), and the map from $\bigcup_{i=1}^{m}\overline{\sigma_{i}}$ to $\bigcup_{i=1}^{m}\overline{\hat{\sigma}_{i}}$ preserves horizontal, resp. vertical, distances between points in $\overline{\sigma_{k}}$ and $\overline{\sigma_{l}}$.
There is one inequality of type 3, resp. 4, if $p\in \overline{\sigma_{k}}\cap \overline{\sigma_{l}}$ and $\overline{\sigma_{k}}$ and $\overline{\sigma_{l}}$ are of any of the first, second, or third forms seen in Figure \ref{fig:cellsofdistancezero} (up to rotation), and the map from $\bigcup_{i=1}^{m}\overline{\sigma_{i}}$ to $\bigcup_{i=1}^{m}\overline{\hat{\sigma}_{i}}$ does not preserve horizontal, resp. vertical, distances between $\overline{\sigma_{k}}$ and $\overline{\sigma_{l}}$, in which case $p$ must be in $\overline{\sigma_{k}}\cap\overline{\sigma_{l}}$.
There are two inequalities of the first two types if $\overline{\sigma_{k}}$ and $\overline{\sigma_{l}}$ are of the fourth form in Figure \ref{fig:cellsofdistancezero} and the map from $\bigcup_{i=1}^{m}\overline{\sigma_{i}}$ to $\bigcup_{i=1}^{m}\overline{\hat{\sigma}_{i}}$ preserves horizontal and vertical distances between points in $\overline{\sigma_{k}}$ and $\overline{\sigma_{l}}$.
If this map only preserves horizontal, resp. vertical, distance then there are two inequalities, one of type 1, resp. 2,  and one of type 4, resp. 3.
If this map preserves neither horizontal nor vertical distance then there are there are two inequalities; namely, types 3 and 4.
We call these the inequalities associated to $\Sigma$.

By Proposition \ref{barycenter in square}, the barycenter $\textbf{m}$ of $\Sigma\subset \big(K_{g,0}(n)\big)^{m} $ is in $SF_{m}\big(K_{g,0}(n)\big)$, so $\hat{\textbf{m}}-\text{bary}(\hat{\textbf{m}})$ must satisfy the inequalities of the first two types associated to $\Sigma$, and $\hat{\textbf{m}}$ must satisfy the inequalities of the second two types.
What's more, $\hat{\textbf{m}}$ must satisfy the all inequalities, those of the first two strictly.
We will use this fact to define our deformation retract.

Let $\hat{\textbf{m}}'$ denote the image of $\hat{\textbf{m}}$ induced by the map $h:\R\to \R$ that sends integers to $0$ and is the identify elsewhere.
We have that $\hat{\textbf{m}}'$ records the dimension of $\Sigma$, as the integer coordinates correspond to cells $\sigma_{i}$ that have $0$ width or height.
Let $\lambda$ be a large enough positive number such that $\frac{1}{\lambda}\hat{\textbf{m}'}$ is in $\big(-\frac{1}{2},\frac{1}{2}\big)^{2m}$, and set this to be $\hat{\textbf{b}}$.
Since $\hat{\textbf{m}}$ strictly satisfies all the inequalities of the first two types associated to $\Sigma$, so does $\hat{\textbf{m}}'$, as the integer coordinates of $\hat{\textbf{m}}'$ cannot be involved in these linear inequalities.
Since $\hat{\textbf{b}}$ is a scaled version of $\hat{\textbf{m}}'$, it must be the case that $-\hat{\textbf{b}}$ violates all the strict inequalities associated to $\Sigma$, those of the first two types because of the minus sign, those of the second two types because of the scaling.
It follows that the point $\hat{\textbf{m}}+\hat{\textbf{b}}\in \hat{\Sigma}$ is a point of $\Sigma$ in $SF_{m}\big(K_{g,0}(n)\big)$, and the point $\hat{\textbf{m}}-\hat{\textbf{b}}\in \hat{\Sigma}$ violates all the inequalities associated to $\Sigma$.
Moving away from this maximally bad point will yield the deformation retraction.

Let $\textbf{m}-\textbf{b}$ denote the point in $\Sigma$ corresponding to $\hat{\textbf{m}}-\hat{\textbf{b}}$.
We claim that the map from $\overline{\Sigma}\cap SF_{m}\big(K_{g,0}(n)\big)$ to $\partial\overline{\Sigma}\cap SF_{m}\big(K_{g,0}(n)\big)$ that pushes every point $\textbf{z}\in \overline{\Sigma}\cap\partial SF_{m}\big(K_{g,0}(n)\big)$ outward along a ray from $\textbf{m}-\textbf{b}$ to until it hits $\partial\overline{\Sigma}$ is a deformation retraction.
To see this, note that the vector from $\hat{\textbf{m}}-\hat{\textbf{b}}$ to $\hat{\textbf{z}}$ is given by $\hat{\textbf{z}}-\hat{\textbf{m}}+\hat{\textbf{b}}$.
It follows that if we move from $\hat{\textbf{z}}$ along this vector for time $t$ we get the point 
\[
\hat{\textbf{z}}_{t}=\hat{\textbf{z}}+t(\hat{\textbf{z}}-\hat{\textbf{m}}+\hat{\textbf{b}}).
\]
This gives us a path in $\overline{\hat{\Sigma}}$ that continuously varies in $\textbf{z}$.
Starting at $\hat{\textbf{z}}$ move along this path until $\hat{\textbf{z}}_{t}\in \partial \hat{\Sigma}$, and then stay put.
It follows that if $\textbf{z}_{t}$ is the point in $SF_{m}\big(K_{g,0}(n)\big)$ corresponding to $\hat{\textbf{z}}_{t}$, this path is contained in $\overline{\Sigma}$, and $\textbf{z}_{t}\in \partial \Sigma$ if and only if $\hat{\textbf{z}}_{t}\in \partial \hat{\Sigma}$.

If we write $T_{\textbf{z}}$ to denote the time such that $\hat{\textbf{z}}_{t}$ is first in $\partial \hat{\Sigma}$, it remains to show that for all $0\le t\le T_{\textbf{z}}$ that $\textbf{z}_{t}\in SF_{m}\big(K_{g,0}(n)\big)$.
By Proposition \ref{local coordinate check boundary}, it suffices to show that the coordinates of $\hat{\textbf{z}}_{t}$ satisfy all the inequalities associated to $\Sigma$.
For the first two types of inequalities, this would follow if
\[
\hat{\textbf{z}}_{t}-\hat{\textbf{m}}=(\hat{\textbf{z}}-\hat{\textbf{m}})+t(\hat{\textbf{z}}-\hat{\textbf{m}}+\hat{\textbf{b}})=(1+t)(\hat{\textbf{z}}-\hat{\textbf{m}})+t\hat{\textbf{b}}
\]
satisfies these inequalities, that is, for the first two types of inequalities it suffices to check that they are satisfied by the local coordinates of $\hat{\textbf{z}}_{t}$.
This is equivalent to showing that $\hat{\textbf{z}}-\hat{\textbf{m}}$ satisfies these inequalities as $t\hat{\textbf{b}}$ must by construction.
This is true, as $\hat{\textbf{z}}-\hat{\textbf{m}}$ are the local coordinates of $\hat{\textbf{z}}$ in $\Sigma$, and this corresponds to a point in $SF_{m}\big(K_{g,0}(n)\big)$ by assumption.
Since these two types of inequalities are linear with no constant terms, we see that if two points in $\R^{2m}$, in our case $\hat{\textbf{z}}-\hat{\textbf{m}}$ and $\hat{\textbf{b}}$, satisfy the first two types of inequalities, then any linear combination of them with positive coefficients will satisfy them as well.

For the inequalities of the second two types, without loss of generality it suffices to consider the inequality of the form $|\hat{x}_{k}+\hat{x}_{l}|\ge 1$.
In this case, the $x$-coordinates of $\hat{m}_{k}$ and $\hat{m}_{l}$ are equal to each other and are equal to $\frac{1}{2}$ (up to sign).
Additionally, the two $2$-vectors of the $x$-coordinates of $\hat{m}_{k}$ and $\hat{m}_{l}$ and of $\hat{b}_{k}$ and $\hat{b}_{l}$ point in the same direction as $(\hat{x}_{k}, \hat{x}_{l})$, which must lie in the first or third quadrant.
It follows that
\begin{align*}
|\hat{x}_{k,t}+\hat{x}_{l, t}|&=\big|\hat{x}_{k}+\hat{x}_{l}+t(\hat{x}_{k}+\hat{x}_{l}-\hat{m}_{x, k}-\hat{m}_{x, l}+\hat{b}_{x, k}+\hat{b}_{x, l})\big|\\
&= |\hat{x}_{k}+\hat{x}_{l}|+t|\hat{x}_{k}+\hat{x}_{l}-\hat{m}_{x, k}-\hat{m}_{x, l}|+t|\hat{b}_{x, k}+\hat{b}_{x, l}|\\
&\ge 1,
\end{align*}
as desired. 
Therefore, $\textbf{z}_{t}$ is in $\overline{\Sigma}\cap SF_{m}\big(K_{g,0}(n)\big)$.

Since the map $(\hat{\textbf{z}}, t)\mapsto \hat{\textbf{z}}_{t}$ is continuous in $\textbf{z}$ and $t$, it induces a deformation retraction of $\overline{\Sigma}\cap SF_{m}\big(K_{g,0}(n)\big)$ onto $\partial\overline{\Sigma}\cap SF_{m}\big(K_{g,0}(n)\big)$.
\end{proof}

Composing the deformation retractions given by the previous lemma produces a deformation retraction of $SF_{m}\big(K(n)\big)$ onto a discrete configuration space.

\begin{T2}
  \thmtextone
\end{T2} 

\begin{proof}
The proofs are identical, so we only prove the first statement.

Order the cells $\Sigma$ of $\big(K_{g,0}(n)\big)^{m}$ that are partially contained in $SF_{m}\big(K_{g,0}(n)\big)$ such that their dimensions are non-increasing.
Then, in order, perform the deformation retraction described in the proof of Lemma \ref{retract partially contained cell no boundary} to these cells.
Concatenating these retractions gives a deformation retraction from $SF_{m}\big(K_{g,0}(n)\big)$ onto $DF_{m}\big(K_{g,0}(n)\big)$.
\end{proof}

Theorem \ref{point is discrete} and Theorem \ref{point is square} combine to describe a discrete combinatorial cell structure for $F_{m}(\Sigma_{g,b})$.

\begin{T3}
  \thmtexttwo
\end{T3} 

We have found a cube complex that is homotopy equivalent to $F_{m}(\Sigma_{g,b})$.
In the next section, we discuss the optimality of our complex, and describe how our method of proof can be used to discretize other configuration spaces.

\section{Remarks and Future Directions}\label{remarks and future directions}

Having constructed a discrete model for $F_{m}(\Sigma_{g,b})$, we conclude this paper by making several remarks about the adaptability of our proof and the optimality of the complex.
We also suggest how this model might be used to study the (co)homology of $F_{m}(\Sigma_{g,b})$.

\subsection{Graph Configuration Spaces}

Recall that our inspiration for constructing a cubical model for surface configuration space comes from the graph setting.
Namely, we were inspired by the following theorem of Abrams:

\begin{T4}
  \thmabrams
\end{T4}

As mentioned in the introduction, Abrams was unable to give an explicit construction of the deformation retraction from $F_{n}(\Gamma)$ onto $DF_{n}(\Gamma)$ due to the nature of his proof, which is entirely different from our proof of Theorem \ref{point is discrete}.
We claim that our proofs of Theorems \ref{point is square} and \ref{square is discrete} can be adapted to the graph setting to construct an explicit deformation retraction.

To see this, note that one can consider the ordered configuration space of squares in a graph $\Gamma$ by defining the open $d$-square---really the open ball of radius $\frac{d}{2}$---centered at $z$ to be the set of points 
\[
\Big\{z'\in \Gamma|d(z, z')<\frac{1}{2}\Big\},
\]
where $d$ is the natural metric on $\Gamma$ that arises by making each open edge isomorphic to the interval $(0,1)$.
Given this notion of a square, we define the \emph{$m^{\text{th}}$-ordered configuration of open $d$-squares in $\Gamma$}, which we denote $SF_{m}(\Gamma)$, to be the space of $m$-tuples of points in $\Gamma$ that are all distance at least $d$ apart from each other, i.e.,
\[
SF_{m}(\Gamma,d):=\{(x_{1}, \dots, x_{m})\in \Gamma^{m}|d(x_{i}, x_{j})\ge d\text{ for all }i\neq j\}.
\]
Note, though $\Gamma$ might have a boundary, i.e., vertices of valence $1$, we do not consider the distance of a point in a configuration to the boundary in our definition of $SF_{m}(\Gamma)$---for graphs this is inconsequential.
As was the case for configurations in the surface $K_{g,b}(n)$, there is a retraction from $F_{m}(\Gamma)$ onto $SF_{m}(\Gamma, 1)$, provided essential vertices are far enough apart and there are no short loops.

\begin{lem}\label{point is square graphs}
Let $\Gamma$ be a connected graph, i.e., a $1$-dimensional cube complex, with at least $n$ vertices.
Then $F_{n}(\Gamma)$ deformation retracts onto $SF_{n}(\Gamma, 1)$ if
\begin{enumerate}
    \item each path connecting distinct essential vertices of $\Gamma$ has length at least $n+1$, and
    \item each homotopically essential path connecting a vertex to itself has length at least $n+1$.
\end{enumerate}
\end{lem}

To prove this lemma it suffices to take analogues of the expansion maps defined in Proposition \ref{move points away from boundary} and Lemma \ref{increase tautological function}, treating the essential vertices of $\Gamma$ like the singularity of $K_{g,0}(n)$.
Doing so gives a map $\phi_{\textbf{z}}$ that increases $\theta_{n}$ on a neighborhood of every configuration $\textbf{z}\in F_{n}(\Gamma)$ such that $\theta_{n}<\frac{n+1}{2n}$.
This map $\phi_{\textbf{z}}$ fixes points at the essential vertices of $\Gamma$, moves nearby points away from the essential vertices, and spreads out tightly bunched collections of points in a configuration.
Proceeding as in Lemma \ref{deformation retract for non crit values}, one can take a finite open cover of the subspace of configurations $\textbf{z}$ such that $\theta^{-1}_{n}\big([r_{1}, r_{2}]\big)$ for $0<r_{1}<r_{2}\le \frac{n+1}{2n}$.
Gluing together the corresponding $\phi_{\textbf{z}}$ via a partition of unity gives rise to a flow on this space that increases $\theta_{n}$ past $\frac{1}{2}<\frac{n+1}{2n}$, yielding the desired deformation retraction.

Alternatively, one can treat the points of a configuration $\textbf{z}\in F_{n}(\Gamma)$ like little balloons that we slowly inflate until they have diameter $1$.
The forces the inflating balloons exert upon each other spread them out, yielding a square configuration space.
Such a map can be formally obtained via the process described in the previous paragraph.
Moreover, one can check that this inflation process allows one to replace $n+1$ with $n-1$ in condition 1 of the lemma.

To complete our constructive proof of Abrams's theorem, it remains to prove that square configuration space deformation retracts onto discrete configuration space.
That is one must prove
\begin{lem}\label{square is discrete graph}
Let $\Gamma$ be a connected graph, i.e., a $1$-dimensional cube complex, with at least $n$ vertices.
Then $SF_{n}(\Gamma,1)$ deformation retracts onto $DF_{n}(\Gamma)$ if
\begin{enumerate}
    \item each path connecting distinct essential vertices of $\Gamma$ has length at least $n+1$, and
    \item each homotopically essential path connecting a vertex to itself has length at least $n+1$.
\end{enumerate}
\end{lem}

Again, one can obtain such a result via the methods of Section \ref{discrete configuration spaces}.
Indeed, Kim used similar techniques to prove a strengthened version of this lemma for configurations of squares on trees that removes condition $1$---condition $2$ is vacuously true in the tree case---\cite[Theorem 3.2.10]{kim2022configuration}.
Namely, one can find inequalities on the local coordinates of the points in a square configuration as in Proposition \ref{local coordinate check boundary}.
Note, that if $\Sigma=\sigma_{1}\times\cdots\times \sigma_{m}$ is partially contained in $SF_{m}(\Gamma)$, then its barycenter $\textbf{m}$ is in $SF_{m}(\Gamma)$, and one can use this fact to find a maximally bad configuration $\textbf{mb}$ in $\Sigma$ as in Lemma \ref{retract partially contained cell no boundary}. 
Pushing any configuration $\textbf{z}$ in $\overline{\Sigma}\cap SF_{n}(\Gamma)$ along the ray starting at $\textbf{mb}$ yields a deformation retraction from $\overline{\Sigma}\cap SF_{n}(\Gamma)$ onto $\partial \overline{\Sigma}\cap SF_{n}(\Gamma)$.
Performing these deformation retractions in sequence as in Theorem \ref{square is discrete} establishes the lemma, and together the deformation retractions of Lemmas \ref{point is square graphs} and \ref{square is discrete graph} yield a deformation retraction of $F_{n}(\Gamma)$ onto $DF_{n}(\Gamma)$, constructively proving Theorem \ref{Abrams theorem}.

Now that we have shown that we can recover an explicit deformation retraction proving Theorem \ref{Abrams theorem}, we use graph intuition to argue that our complexes $K_{g,b}(n)$ are optimal, i.e., are as small as possible for such a discretization.

\subsection{Optimality}
Though we do not give a rigorous proof that our constraint that $m\le n$ is optimal, we give a heuristic argument inspired by graph configuration spaces for the sharpness of this bound.
This argument also suggests that similar bounds would be needed for any model for $F_{m}(\Sigma)$ arising from a square-tiling of a surface $\Sigma$.

In Abrams's Theorem 2.1 \cite{abrams2000configuration} there are two conditions on the size of $\Gamma$:
\begin{enumerate}
\item each path connecting distinct essential vertices of $\Gamma$ has length at least $n+1$, and
    \item each homotopically essential path connecting a vertex to itself has length at least $n+1$.
\end{enumerate}
While the first condition can be tightened---see Kim, Ko, and Park \cite[Theorem 2.4]{ko2012characteristics}, Prue and Scrimshaw \cite[Theorem 3.2]{prue2014abrams}, and the previous subsection---the second cannot.
To see this, let $C_{n}$ denote cycle graph on $n$ vertices.
The complex $DF_{n}(C_{n})$ consists of $n!$ disjoint $0$-cells, one for each way to number the vertices of $C_{n}$.
On the other hand, $F_{n}(C_{n})$ is homotopy equivalent to $(n-1)!$ disjoint circles, corresponding to the $(n-1)!$ ways to place $n$ labeled points on a circle.
The two spaces are not homotopy equivalent, and inclusion of $C_{n}$ into an arbitrary graph $\Gamma$ shows the strictness of the second condition.

This problem of lacking a free higher dimensional cell to move a point from one $0$-cell to another, rears it head in the surface setting as well.
To see this, consider the $n-1$ cycle in $F_{n}(\R^{2})$ arising from a solar system on $n$ planets, i.e., the cycle that corresponds to the points $z_{1}$ and $z_{2}$ orbiting each other, the point $z_{3}$ orbiting this pair independently, the point $z_{4}$ orbiting this triple independently, and so on and so forth.
Converting these points to unit-squares, we see that the corresponding cycle needs at least an $n\times n$ square of space.
It follows that the centers of these squares must occupy at least an $(n-1)\times (n-1)$ square.
With this in mind, we note that in order to move this square about $K$, the open $1$-neighborhood about such an $(n-1)\times (n-1)$ square must be contractible or else one would run into a problem akin to that of configurations on the cycle graph.
As such, we need the injectivity radius of a model to be at least $\frac{n+1}{2}$.
Our models $K_{g,0}(n)$ satisfy this requirement optimally; moreover, one can check that $K_{g,0}(0)$ is a minimal tiling of $\Sigma_{g, 0}$, proving that $K_{g,0}(n)$ is an optimal starting point.

Replacing $K_{g,0}(0)$ with a model $K'_{g,0}(0)$ that has more singular points might allow one to use a coarser refinement $K'_{g,0}(n')$ when constructing a discrete model for $F_{m}(\Sigma_{g,0})$, as the added singularities might obviate any concerns about short geodesics.
Indeed, we suspect that distinct singularities in $K'_{g,0}(n')$ would only need to be a Chebyshev distance $n-1$ apart,
so it might be the case that $F_{m}(\Sigma_{g,0})$ is homotopy equivalent to $DF_{m}\big(K'_{g,0}(n-2)\big)$.
That said, $K'_{g,0}(0)$ would have more $2$-cells than $K_{g,0}(0)$, and the extra efficiency gained by a marginally coarser refinement does not outweigh this initial cost.
As such, we believe that the square-tiled surfaces $K_{g,0}(n)$ and $K_{g,1}(n)$ yield optimal models given our method of construction.

\subsection{Higher Dimensions}

Given Abrams's and our models, we believe that similar models should exist for the ordered configuration space of points in any pure dimensional cube complex with a well-defined frame---one should convince themselves that we can obtain models for $F_{m}(\Sigma_{g,b})$, with $b$ any non-negative integer, by sufficiently subdividing $K_{g,0}(n)$ and removing $2$-cells.
As such, we make the following conjecture:

\begin{conj}\label{higher dimensions}
Let $K$ be a cube complex such that if $i\le k$, then every $i$-cube in $K$ is the face of some $k$-cube.
Moreover, assume that $K$ is framed away from some codimension-$2$ subspace.
If $K(n)$ denotes the $n$-fold subdivision of $K$, then, for all $m\le n$, the ordered configuration space of $m$ points in $K(n)$ is homotopy equivalent to the $m^{\text{th}}$-discrete ordered configuration space of $K(n)$.
\end{conj}

We believe that Conjecture \ref{higher dimensions} is a reasonable starting point, though we believe that these constraints on $K$ can be weakened substantially.
Indeed, we claim that the methods of this paper can be used to discretize $F_{m}(S^{2})$ and $F_{m}(N_{g,0})$, where $N_{g,0}$ is the non-orientable surface of genus $g$.
Though cubulations of $N_{g,0}$ cannot be framable, and many of the naive cubulations of $S^{2}$ like the cube, are not framable even away from a set of $0$-cells, there is a method of subdividing the $2$-cells of such a complex to get a complex with well-defined notions of horizontal and vertical away from the vertices.
The fact that one cannot distinguish left from right, or up from down on large subsets of these cube complexes is not an issue to running our machinery.
As such, one can discretize $F_{m}(S^{2})$ and $F_{m}(N_{g,0})$, and this method of preemptively subdividing should be extendable to higher dimensional complexes, yielding a way to discretize the configuration spaces of many manifolds.

\subsection{Homology Calculations and Asymptotics}
Finally, we propose using discrete Morse theory on our model to study the (co)homology groups of surface configuration spaces.
We do so in light of the success of discrete Morse theory in the study of graph and disk configuration spaces.

In the study of graph configuration spaces, perhaps the most used tool is the discrete gradient vector field on Abrams's discrete configuration space constructed by Farley and Sabalka in \cite{farley2005discrete}.
This relatively simple construction has been used numerous times; we list just a few instances:
Farley and Sabalka used this discrete gradient vector field to partially describe the cohomology ring of the unordered configuration space of a tree in \cite{farley2008cohomology}, and Gonz\'{a}lez and Hoekstra-Mendoza furthered these computations in \cite{gonzalez2022cohomology}.
In \cite{ko2012characteristics}, Ko and Park used this discrete gradient vector field to compute the first homology of any unordered graph configuration space, while also determining the first homology of the ordered configuration space of $2$ points in any planar graph.
Scheirer used this discrete gradient vector field to compute the topological complexity of certain (un)ordered graph configuration spaces in \cite{scheirer2018topological}, and  Aguilar-Guzm\'{a}n, Gonz\'{a}lez, and Hoekstra-Mendoza \cite{aguilar2022farley} extended these calculations, finding the sequential topological complexities of ordered graph configuration spaces.
Additionally, Farley and Sabalka used their discrete gradient vector field to find presentations for graph braid groups in \cite{farley2012presentations}, and Ramos \cite{ramos2018stability} further used their discrete gradient vector field to prove that the homologies of these braid groups stabilize.

Discrete gradient vector fields have also been used to great success in the study of disk configuration spaces. 
Alpert found a discrete gradient vector field on $\text{cell}(n,2)$, a discrete combinatorial model for the ordered configuration space of $n$ open unit-diameter disks in the infinite strip of width $2$, and used it to compute a basis for the homology of this space in \cite{alpert2020generalized}.
Building on this, Alpert and Manin \cite{alpert2021configuration1} used discrete Morse theory to find a basis for the homology of the (un)ordered configuration space of open unit-diameter disks in an infinite strip of width $w$.
Discrete Morse theory was also used by Alpert, Kahle, and MacPherson to the calculate asymptotic Betti numbers for the ordered configuration space of open disks in the infinite strip of width $w$ in\cite{alpert2021configuration}, as well as bounds for the asymptotic Betti numbers of the ordered configuration space of square in a rectangle in \cite{alpert2024asymptotic}.
Alpert, Bauer, Kahle, MacPherson, and Spendlove \cite{alpert2023homology} furthered these calculations via discrete Morse theory.

Given the fruitfulness of discrete Morse theory in the study of graph and disk configuration spaces, we ask the following questions:

\begin{question}
Can one provably construct maximal discrete gradient vector fields on the discrete configuration spaces $DF_{m}\big(K_{g,0}(n)\big)$ and $DF_{m}\big(K^{*}_{g,1}(n)\big)$?
\end{question}

Recent work of Gonz\'{a}lez and Gonz\'{a}lez \cite{gonzalez2023algorithmic}
gives hope that such a maximal discrete gradient vector field is not far from reach.
Namely, they describe a method for constructing maximal discrete gradient vector fields on the discrete ordered configuration space of a simplical complex.
Can their methods be adapted to cube complex setting?

Given such a discrete gradient vector field, we hope to use it to compute the asymptotic Betti numbers of ordered surface configuration spaces.
Church, Ellenberg, and Farb proved that Betti numbers of the ordered configuration space of points in a connected non-compact manifold of finite-type of dimension at least $2$ behave polynomially in the number of points, i.e., $\beta_{i}\big(F_{m}(M)\big)=p(m)$, where $p$ is a polynomial of degree at most $2i$ \cite[Theorem 6.4.3]{church2015fi}.
That said, we do not know the order of this polynomial $p(m)$ or the leading coefficient of $p(m)$ for most manifolds $M$, e.g., see \cite[Conjecture 2.10]{pagaria2022asymptotic}, let alone explicit formulae for $\beta_{i}\big(F_{m}(M)\big)$.
With this in mind, we ask

\begin{question}
What is the leading coefficient of $\beta_{i}\big(F_{m}(\Sigma_{g,1})\big)$ as a function of $m$?
\end{question}

We hope that our discrete models yield new approaches to these questions, and give new insights on surface configuration spaces.

\bibliographystyle{amsalpha}
\bibliography{DiscreteConf}

\end{document}